\documentclass[11pt]{article}

\usepackage[T1]{fontenc}
\usepackage[utf8]{inputenc}
\usepackage{lmodern}
\usepackage{amsmath,amssymb,amsthm}
\usepackage{hyperref}
\usepackage[a4paper,margin=1in]{geometry}

\usepackage{authblk}

\usepackage{float}
\usepackage{tikz}
\usetikzlibrary{arrows.meta,calc,positioning}

\newcommand{\keywords}[1]{%
  \par\bigskip
  \noindent\textbf{Keywords.} #1\par
}

\newcommand{\subjclass}[2][]{%
  \par\bigskip
  \noindent\textbf{MSC#1.} #2\par
}

\title{\bf An Introduction to Torsors in Mathematics\\
with a View Toward $\Sigma$-Protocols in Cryptography}

\author{\Large Takao Inou\'{e}}
\affil{\large Faculty of Informatics, Yamato University, \\ Osaka, Japan\footnote{Email: inoue.takao@yamato-u.ac.jp; \\ Personal Email: takaoapple@gmail.com \\ [I prefer my personal email address for correspondence.]}}
\date{March 12, 2026}

\newtheorem{definition}{Definition}[section]
\newtheorem{remark}[definition]{Remark}
\newtheorem{example}[definition]{Example}
\newtheorem{proposition}[definition]{Proposition}

\begin{document}

\maketitle

\begin{abstract}
This paper provides a preparatory introduction to torsors,
written with a view toward later applications in the author's work.
Rather than aiming at a comprehensive survey,
the exposition focuses on those aspects of torsors
that are most useful for understanding torsor-based reasoning:
group actions, orbits, free transitive actions,
the absence of a canonically chosen origin,
and the interpretation of group elements as transports between points.
After developing the basic definition and several elementary examples,
we emphasize a central theme:
torsors are not only characterized abstractly by free transitive group actions,
but also arise naturally as objects obtained by gluing local trivial pieces
by means of transition data satisfying cocycle conditions.
A brief optional section indicates a sheaf- and topos-theoretic perspective.
In the final part, we explain how these ideas prepare the ground
for later conceptual applications, including aspects of $\Sigma$-protocols.
\end{abstract}



\tableofcontents

\section{Introduction}

The notion of a torsor appears in many parts of modern mathematics,
but it is often first encountered only after one has already learned
a substantial amount of algebra, geometry, or cohomology.
As a result, the concept can seem more technical or mysterious
than it really is.
The purpose of these notes is to present torsors in a way
that is elementary enough to be read independently,
while at the same time emphasizing those structural features
that become important in later applications.

At an intuitive level, a torsor is a space with symmetry
but without a chosen origin.
A group has a distinguished identity element,
and so its elements may be viewed in absolute terms.
A torsor, by contrast, carries the same pattern of symmetry
without coming equipped with any preferred basepoint.
What remains meaningful is relative position:
given two points of a torsor,
there is a unique group element that transports one point to the other.
Thus, whereas groups naturally support an absolute language,
torsors are governed by transport, comparison, and displacement.

The basic geometric model is provided by affine spaces.
An affine space resembles a vector space,
but it has no canonically chosen zero vector.
One may subtract two points and obtain a vector,
but one may not intrinsically add two points.
This already exhibits the central torsor phenomenon:
differences are meaningful, but absolute coordinates require an arbitrary choice.
This affine viewpoint is not merely a geometric convenience.
It brings into focus a structural principle that is important far beyond
elementary geometry:
the mathematically significant objects are often not points in isolation,
but differences, displacements, and translations.
In this sense, affine spaces provide the clearest first model
for the general theory of torsors.

This perspective is especially useful in analysis.
Many basic analytic constructions are naturally expressed
in terms of increments such as
\[
f(x+h)-f(x),
\]
where the essential object is the displacement $h$
rather than the absolute position of the point $x$.
From this point of view,
the domain is naturally regarded as an affine space,
while increments belong to the associated vector space.
Thus one distinguishes clearly between \emph{where one is}
and \emph{how one moves}.
This style of thinking is very much in the spirit
of the affine viewpoint in analysis emphasized by Schwartz \cite{Schwartz}.
For the purposes of the present notes,
it also explains why affine spaces are the most natural entry point
to torsors.

The present paper is not intended as a comprehensive survey
of all appearances of torsors in mathematics.
Its aim is more focused.
It is written as preparatory mathematical background
for the sheaf-theoretic and torsor-theoretic viewpoint
developed in the author's later work \cite{Inoue2026},
especially for situations in which torsor-based reasoning appears
in a sheaf-theoretic or local-to-global form.
In such settings,
the main issue is often not the existence
of a globally distinguished point,
but rather the possibility of making compatible local choices
and understanding the obstructions to global ones.
For that reason, these notes place particular emphasis
on local triviality, gluing data, cocycle descriptions,
and the distinction between local sections and global sections.

This point of view is especially important
for the conceptual interpretation of certain protocol-theoretic constructions.
In later work, one considers sheaf-theoretic objects associated with
$\Sigma$-protocols and studies situations in which such objects
behave like torsors under suitable group actions.
From that perspective, local triviality is related to simulation,
while the failure of appropriate global choices
reflects structural security constraints.
The purpose of the present paper is not to establish those later results,
but to supply the mathematical language needed to understand them.
In particular, the reader should come away with a clear understanding of
what a torsor is,
how torsors arise from gluing data,
how sheaf-theoretic torsors differ from set-theoretic ones,
and why the contrast between local and global sections matters.

Accordingly, one of the central themes of the paper is that
torsors should not be understood only through the abstract definition
of a free and transitive group action.
That definition is fundamental and will be developed carefully,
but it tells only one part of the story.
Equally important is the fact that torsors arise naturally
from local trivial pieces that are glued together by transition data.
Locally, a torsor may look indistinguishable from the acting group itself.
Globally, however, these local trivializations need not combine
into a single preferred global one.
The mismatch between local trivializations is recorded
by transition functions satisfying cocycle conditions.
In this way, torsors become among the simplest and most instructive examples
of mathematical objects assembled from compatible local data.

This gluing viewpoint places torsors close to broader themes
in geometry, descent, and cohomology.
One of the reasons torsors are so useful
is that they stand at the meeting point of group actions,
local-to-global reconstruction,
and first cohomological classification.
Even when one does not pursue the full theory,
it is important to see that torsors belong not merely
to elementary group theory,
but to the general mathematics of compatibility and reconstruction.
For this reason, after the basic sections of the paper,
we include a short optional section
indicating a sheaf- and topos-theoretic perspective.
Readers mainly interested in the elementary theory
may safely skip that section on a first reading.

The paper is organized as follows.
Section~2 reviews the language of group actions and orbits.
Section~3 studies free and transitive actions,
which together lead directly to the torsor concept.
Section~4 gives the formal definition of a torsor.
Section~5 develops the main elementary examples,
with particular emphasis on affine spaces.
Section~6 studies basepoints, differences, and transporters,
highlighting the non-canonical character of identifications with groups.
Section~7 contains exercises on the basic theory.
Section~8 introduces local triviality and gluing data.
Section~9 develops cocycle descriptions and reconstruction.
Section~10 turns to sheaf torsors
and clarifies the role of local and global sections.
Section~11 is a logically optional section
offering a sheaf- and topos-theoretic glimpse,
though it is strongly recommended
for readers interested in the later applications.
Section~12 indicates how the preceding material prepares the ground
for later applications to $\Sigma$-protocols.
We conclude in Section~13 with a brief summary.

The appendices provide detailed solutions to the exercises in Section~7,
written in a way that emphasizes the use of the basic definitions,
a suggested lecture plan for using these notes in a classroom setting
or as preparation for a further reading course on the author's later paper,
and a short glossary of basic terms for quick reference.

\section{Group Actions and Orbits}

We begin with the elementary language of group actions.
Although torsors are often introduced as a special kind of group action,
it is useful to proceed in stages.
The notions of orbit, stabilizer, and transport between points
already contain much of the structure that later becomes central.

\begin{definition}
Let $G$ be a group and let $X$ be a set.
A \emph{left action} of $G$ on $X$ is a map
\[
G\times X \to X,\qquad (g,x)\mapsto g\cdot x,
\]
such that
\begin{enumerate}
\item $e\cdot x=x$ for all $x\in X$, where $e$ is the identity element of $G$,
\item $(gh)\cdot x=g\cdot (h\cdot x)$ for all $g,h\in G$ and all $x\in X$.
\end{enumerate}
In this situation, we say that $G$ \emph{acts on} $X$,
or that $X$ is a \emph{$G$-set}.
\end{definition}

\begin{remark}
A right action may be defined similarly.
In these notes we work primarily with left actions,
since they are sufficient for the basic theory
and for the conceptual applications that follow.
\end{remark}

A group action allows one to regard the elements of $G$
as transformations of the set $X$.
In this way, the abstract algebraic structure of the group
becomes visible through the movement of points in $X$.
For the purposes of torsor theory,
the important issue is not merely that points move,
but how they move relative to one another.

\begin{definition}
Let $G$ act on a set $X$, and let $x\in X$.
The \emph{orbit} of $x$ is the subset
\[
G\cdot x:=\{g\cdot x\mid g\in G\}\subseteq X.
\]
\end{definition}

Thus the orbit of a point consists of all points
that can be reached from it by the action of the group.
Orbits decompose the set $X$ into regions of mutual accessibility:
two points lie in the same orbit precisely when one can be transported
to the other by some element of $G$.

\begin{definition}
Let $G$ act on a set $X$, and let $x\in X$.
The \emph{stabilizer} of $x$ is the subgroup
\[
G_x:=\{g\in G\mid g\cdot x=x\}.
\]
\end{definition}

The stabilizer measures the symmetry that remains invisible at the point $x$:
its elements act, but leave $x$ fixed.
From the viewpoint of torsors,
stabilizers will eventually disappear,
because the actions we care about most will have no nontrivial fixed symmetries.

\begin{example}
Let $V$ be a vector space over a field $k$.
Then the additive group $(V,+)$ acts on the underlying set of $V$ by translation:
\[
v\cdot x:=v+x.
\]
For any $x\in V$, the orbit of $x$ is all of $V$,
since every point can be reached from every other by translation.
Moreover, the stabilizer of $x$ is trivial,
because $v+x=x$ implies $v=0$.
This basic example already foreshadows the torsor situation.
\end{example}

\begin{example}
Let $G$ act on itself by left multiplication:
\[
g\cdot x:=gx.
\]
Then the orbit of any point is all of $G$,
and the stabilizer of every point is trivial.
This is the model example of a free and transitive action,
to which we will return in the next section.
\end{example}

The orbit viewpoint is especially useful
because it shifts attention away from the set $X$ as a whole
and toward the relation among its points.
If $x,y\in X$ lie in the same orbit,
there exists at least one element $g\in G$ such that
\[
g\cdot x=y.
\]
In general such an element need not be unique.
The failure of uniqueness is controlled by stabilizers.
Indeed, if $g\cdot x=y$ and $h\cdot x=y$,
then
\[
h^{-1}g\cdot x=x,
\]
so $h^{-1}g\in G_x$.
Thus uniqueness of transport is closely related
to the triviality of stabilizers.

This observation already points toward torsors.
A torsor is, roughly speaking, a space in which every point can be reached
from every other by a unique group element.
In other words, it is a setting in which the group action supplies
a perfectly rigid notion of relative position.
Before making this precise, however,
we isolate the two relevant properties separately:
freedom and transitivity.

\section{Free and Transitive Actions}

We now discuss the two key properties
that distinguish the actions underlying torsors.
Each property has a simple intuitive meaning.
Transitivity says that the action has only one orbit:
every point can be moved to every other point.
Freedom says that no nontrivial group element fixes a point:
the acting group leaves no hidden symmetry at any location.
When both conditions hold simultaneously,
the action behaves as rigidly as possible.

\begin{definition}
Let $G$ act on a set $X$.
\begin{enumerate}
\item The action is called \emph{transitive} if for every $x,y\in X$,
there exists $g\in G$ such that
\[
g\cdot x=y.
\]
Equivalently, $X$ consists of a single orbit.

\item The action is called \emph{free} if for every $x\in X$ and every $g\in G$,
\[
g\cdot x=x \quad \Longrightarrow \quad g=e.
\]
Equivalently, the stabilizer of every point is trivial.
\end{enumerate}
\end{definition}

\begin{remark}
These two conditions are logically independent.
An action may be transitive without being free,
and it may be free without being transitive.
Torsors arise precisely when both hold at once.
\end{remark}

\begin{example}
Let $G$ act on the set of left cosets $G/H$ by left multiplication:
\[
g\cdot (aH):=(ga)H.
\]
This action is transitive,
since any coset can be moved to any other by an appropriate group element.
However, it is not free unless $H=\{e\}$,
because the stabilizer of the coset $H$ is exactly $H$.
\end{example}

\begin{example}
Let $G=\mathbb{Z}$ act on $X=\mathbb{Z}$ by translation:
\[
n\cdot x:=n+x.
\]
This action is free and transitive.
Hence $\mathbb{Z}$, viewed as a set acted on by itself,
already provides an example of the kind of structure
that will later be abstracted into the notion of a torsor.
\end{example}

\begin{proposition}
Let $G$ act on a set $X$.
Then the following are equivalent:
\begin{enumerate}
\item the action is free and transitive;
\item for every $x,y\in X$, there exists a unique $g\in G$ such that
\[
g\cdot x=y.
\]
\end{enumerate}
\end{proposition}

\begin{proof}
Assume first that the action is free and transitive.
Let $x,y\in X$.
By transitivity, there exists $g\in G$ such that $g\cdot x=y$.
Suppose also that $h\cdot x=y$.
Then
\[
h^{-1}g\cdot x=x.
\]
Since the action is free, we must have $h^{-1}g=e$,
and therefore $g=h$.
Thus the required element is unique.

Conversely, assume that for every $x,y\in X$
there exists a unique $g\in G$ such that $g\cdot x=y$.
Existence immediately implies transitivity.
To prove freeness, let $x\in X$ and suppose that $g\cdot x=x$.
But also $e\cdot x=x$.
By uniqueness, $g=e$.
Hence the action is free.
\end{proof}

This proposition is fundamental.
It shows that a free and transitive action
is exactly the same thing as a system in which
every relative displacement between two points
is encoded by a unique element of the acting group.
For later purposes, it is useful to isolate this group element explicitly.

\begin{definition}
Suppose that $G$ acts freely and transitively on $X$.
For $x,y\in X$, let
\[
\operatorname{tr}(x,y)\in G
\]
denote the unique element such that
\[
\operatorname{tr}(x,y)\cdot x=y.
\]
We call $\operatorname{tr}(x,y)$ the \emph{transporter}
from $x$ to $y$.
\end{definition}

Figure~\ref{fig:transporter-basic} illustrates the basic transporter idea:
the relation between two points is measured by a unique group element.

\begin{figure}[H]
\centering
\begin{tikzpicture}[>=Latex, baseline=(current bounding box.center)]
  \node[circle, fill=black, inner sep=1.5pt, label=below:$x$] (x) at (0,0) {};
  \node[circle, fill=black, inner sep=1.5pt, label=below:$y$] (y) at (4,0) {};
  \draw[->, thick] (x) -- node[above] {$\operatorname{tr}(x,y)$} (y);
\end{tikzpicture}
\caption{The transporter from $x$ to $y$ is the unique group element
sending $x$ to $y$.}
\label{fig:transporter-basic}
\end{figure}
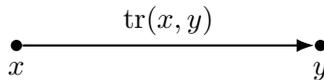

\begin{remark}
The transporter notation is not universal,
but the underlying idea is extremely useful.
In a torsor, the group is recovered not from distinguished points,
but from the system of transports between points.
This will later make it natural to think in terms of relative position
rather than absolute coordinates.
\end{remark}

The transporter enjoys the formal properties one would expect.
For any $x,y,z\in X$, one has
\[
\operatorname{tr}(x,x)=e,
\qquad
\operatorname{tr}(y,z)\operatorname{tr}(x,y)=\operatorname{tr}(x,z),
\qquad
\operatorname{tr}(x,y)^{-1}=\operatorname{tr}(y,x).
\]
These identities are immediate from uniqueness.
They show that,
although the space $X$ itself need not carry a group law,
the comparison of its points is governed by the group $G$.

Figure~\ref{fig:transporter-composition-early} visualizes
the composition law for transporters.

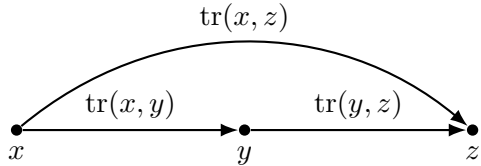
\begin{figure}[H]
\centering
\begin{tikzpicture}[>=Latex, baseline=(current bounding box.center)]
  \node[circle, fill=black, inner sep=1.5pt, label=below:$x$] (x) at (0,0) {};
  \node[circle, fill=black, inner sep=1.5pt, label=below:$y$] (y) at (3,0) {};
  \node[circle, fill=black, inner sep=1.5pt, label=below:$z$] (z) at (6,0) {};

  \draw[->, thick] (x) -- node[above] {$\operatorname{tr}(x,y)$} (y);
  \draw[->, thick] (y) -- node[above] {$\operatorname{tr}(y,z)$} (z);
  \draw[->, thick, bend left=40] (x) to node[above] {$\operatorname{tr}(x,z)$} (z);
\end{tikzpicture}
\caption{Composition of transporters:
$\operatorname{tr}(y,z)\operatorname{tr}(x,y)=\operatorname{tr}(x,z)$.}
\label{fig:transporter-composition-early}
\end{figure}

\begin{proposition}
Let $G$ act freely and transitively on a nonempty set $X$,
and choose a point $x_0\in X$.
Then the map
\[
\phi_{x_0}:G\to X,\qquad g\mapsto g\cdot x_0,
\]
is a bijection.
\end{proposition}

\begin{proof}
Surjectivity follows from transitivity:
every point of $X$ has the form $g\cdot x_0$ for some $g\in G$.
Injectivity follows from freeness:
if $g\cdot x_0=h\cdot x_0$, then
\[
h^{-1}g\cdot x_0=x_0,
\]
hence $h^{-1}g=e$, so $g=h$.
\end{proof}

This proposition explains why a torsor may be informally described
as a group without a chosen identity element.
Once a point $x_0$ is selected,
the set $X$ becomes identified with $G$.
But that identification depends on the choice of $x_0$,
and there is generally no canonical reason to prefer one point over another.
Thus the group structure is present only through transport,
not through an intrinsically distinguished origin.

This non-canonical character is not an accidental inconvenience;
it is the essence of the torsor idea.
A torsor is not merely a copy of a group.
Rather, it is a space that locally or relationally behaves like a group,
while globally lacking a distinguished identity.
That is why torsors arise so naturally in contexts
where local choices are available but global choices are not.

The next section formalizes these ideas
by introducing the definition of a torsor.

\section{The Definition of a Torsor}

We are now ready to isolate the central notion of these notes.
Informally, a torsor is a space on which a group acts
freely and transitively.
Equivalently, it is a space in which every point can be transported
to every other point by a unique group element.
This makes a torsor look very much like the group itself,
except that no identity element is singled out in advance.

\begin{definition}
Let $G$ be a group.
A \emph{left $G$-torsor} is a nonempty set $X$
equipped with a left action of $G$ on $X$
such that the action is free and transitive.
\end{definition}

Unless otherwise stated, the word \emph{torsor}
will mean a left torsor in the set-theoretic sense.
Later we will consider sheaf-theoretic versions,
but the elementary case of sets is the first model to keep in mind.

By Proposition~3.3, the preceding definition admits
an equivalent and very useful reformulation.

\begin{proposition}
Let $G$ be a group acting on a nonempty set $X$.
Then the following are equivalent:
\begin{enumerate}
\item $X$ is a $G$-torsor;
\item for every $x,y\in X$, there exists a unique $g\in G$ such that
\[
g\cdot x=y.
\]
\end{enumerate}
\end{proposition}

\begin{proof}
This is exactly Proposition~3.3,
restated in the language of torsors.
\end{proof}

Thus a torsor is characterized by unique transport.
This is one of the reasons the notion is so flexible:
instead of thinking of the group as a set of abstract symmetries,
one may think of it as the system of all relative displacements
between points of the torsor.

\begin{remark}
A torsor is often called a \emph{principal homogeneous space}.
The adjective ``homogeneous'' refers to the fact
that every point looks the same from the viewpoint of the action,
while ``principal'' reflects the absence of nontrivial stabilizers.
We will return to this terminology in Section~7.
\end{remark}

\subsection{Torsors as groups without chosen identity}

A group acts on itself by left multiplication,
and this action is free and transitive.
Hence every group $G$ determines a canonical example of a $G$-torsor,
namely the underlying set of $G$ itself.
What distinguishes a general torsor from the group acting on itself
is that a general torsor does not come with a preferred point
corresponding to the identity element.

\begin{proposition}
Let $X$ be a $G$-torsor, and choose a point $x_0\in X$.
Then the map
\[
\phi_{x_0}:G\to X,\qquad g\mapsto g\cdot x_0,
\]
is a bijection.
\end{proposition}

\begin{proof}
This was proved in Proposition~3.5.
\end{proof}

The significance of this bijection is conceptual rather than merely formal.
It shows that every torsor becomes indistinguishable from the group
once a basepoint has been chosen.
However, the choice of basepoint is external to the torsor structure itself.
Different choices lead to different identifications,
and there is generally no canonical one.
Thus a torsor is not literally a group,
but rather a space that becomes group-like after a choice is made.

\begin{remark}
This is why one often hears the slogan
that a torsor is ``a group without a chosen identity element.''
The slogan is useful, provided one remembers
that a torsor is not itself equipped with a multiplication law.
What it possesses is not an internal group structure,
but an external simply transitive action of a group.
\end{remark}

\subsection{Transport and relative position}

If $X$ is a $G$-torsor,
then for every pair of points $x,y\in X$
there is a unique transporter
\[
\operatorname{tr}(x,y)\in G
\]
such that
\[
\operatorname{tr}(x,y)\cdot x=y.
\]
This perspective allows one to think of the acting group
as encoding all relative positions in $X$.

Figure~\ref{fig:transporter-relative-position} emphasizes this point:
a torsor is governed by transport between points
rather than by absolute coordinates.

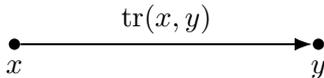
\begin{figure}[H]
\centering
\begin{tikzpicture}[>=Latex, baseline=(current bounding box.center)]
  \node[circle, fill=black, inner sep=1.5pt, label=below:$x$] (x) at (0,0) {};
  \node[circle, fill=black, inner sep=1.5pt, label=below:$y$] (y) at (4,0) {};
  \draw[->, thick] (x) -- node[above] {$\operatorname{tr}(x,y)$} (y);
\end{tikzpicture}
\caption{In a torsor, relative position is measured by the unique transporter
between two points.}
\label{fig:transporter-relative-position}
\end{figure}

The identities
\[
\operatorname{tr}(x,x)=e,
\qquad
\operatorname{tr}(y,z)\operatorname{tr}(x,y)=\operatorname{tr}(x,z),
\qquad
\operatorname{tr}(x,y)^{-1}=\operatorname{tr}(y,x)
\]
show that the transporter behaves like a difference operation.
In an affine space, for example,
the ``difference'' between two points is a vector.
A torsor generalizes this phenomenon:
the difference between two points is measured not in a vector space,
but in the acting group.

This viewpoint is particularly important for the later sections.
When torsors arise from local triviality and gluing data,
or when they are considered in a sheaf-theoretic setting,
it is often the relative comparison of local choices
that matters most.
Thus the transporter formalism is not merely convenient notation;
it expresses the basic logic of the subject.

\subsection{Right torsors}

For completeness, we record the right-handed version of the notion.

\begin{definition}
Let $G$ be a group.
A \emph{right $G$-torsor} is a nonempty set $X$
equipped with a right action
\[
X\times G\to X,\qquad (x,g)\mapsto x\cdot g,
\]
such that the action is free and transitive.
\end{definition}

Everything said above has an evident right-handed analogue.
In these notes, left torsors will remain the default convention,
since they are sufficient for the intended applications.
However, in geometry and descent theory
both left and right torsors arise naturally,
and one must pay attention to conventions.

\subsection{The importance of nonemptiness}

The nonemptiness condition in the definition of a torsor
is not a mere technicality.
If $X$ were allowed to be empty,
then the action would vacuously satisfy several formal properties,
but the fundamental intuition of a torsor would be lost.
A torsor is meant to model a space in which points can be compared
by unique transports.
Without points, there is nothing to transport.

Moreover, in later sheaf-theoretic contexts,
the distinction between local nonemptiness and global nonemptiness
becomes essential.
A sheaf torsor may be locally inhabited
without admitting any global section.
This local/global tension is one of the main reasons
the notion becomes mathematically powerful.

\subsection{First summary}

Let us summarize the basic picture.
A $G$-torsor is a nonempty set $X$
on which $G$ acts freely and transitively.
Equivalently, every pair of points in $X$
determines a unique transporter in $G$.
Choosing a basepoint identifies $X$ with $G$,
but only non-canonically.
Thus a torsor is group-like without carrying a preferred identity.

This is the elementary definition.
Its importance, however, lies not only in the formal condition
of free transitivity,
but also in the way it prepares us for later constructions.
A torsor is the simplest setting in which
one has coherent relative position without absolute origin.
For that reason, torsors arise naturally in affine geometry,
in gluing problems,
in descent-theoretic situations,
and, as we shall later suggest,
in sheaf-theoretic approaches to $\Sigma$-protocols.

The next section develops several basic examples,
which should be kept in mind throughout the rest of the paper.

\section{Basic Examples of Torsors}

The abstract definition of a torsor becomes much clearer
once one has several concrete examples in mind.
Among these, affine spaces are by far the most important for intuition.
They show in the simplest possible way
how one may have a meaningful notion of difference
without having a canonically chosen origin.
After discussing affine spaces first,
we consider several further examples
that illustrate the same pattern in algebraic settings.

\subsection{Affine spaces}

Affine spaces provide the basic geometric model of a torsor.
Informally, an affine space is like a vector space
from which the origin has been forgotten.
One may compare points by taking their difference,
but one does not intrinsically add points to points.

\begin{definition}
Let $V$ be a vector space over a field $k$.
An \emph{affine space modeled on $V$}
is a nonempty set $A$
equipped with a map
\[
A\times V \to A,\qquad (a,v)\mapsto a+v,
\]
such that:
\begin{enumerate}
\item $a+0=a$ for all $a\in A$,
\item $(a+v)+w=a+(v+w)$ for all $a\in A$ and $v,w\in V$,
\item for every $a,b\in A$,
there exists a unique $v\in V$ such that
\[
a+v=b.
\]
\end{enumerate}
\end{definition}

The third condition says exactly
that the action of the additive group of $V$ on $A$
is free and transitive.
Thus an affine space modeled on $V$
is precisely a $V$-torsor,
where $V$ is regarded as an additive group.

\begin{proposition}
Let $V$ be a vector space.
An affine space modeled on $V$
is the same thing as a torsor under the additive group $(V,+)$.
\end{proposition}

\begin{proof}
The axioms of an affine space state precisely
that the translation action of $(V,+)$ on $A$
is defined, free, and transitive.
Conversely, any torsor under the additive group $(V,+)$
may be written in the form $a+v$,
and then satisfies the affine-space axioms.
\end{proof}

This example is fundamental enough
that one may view torsors as a direct generalization of affine spaces.
In an affine space, the quantity relating two points
is a vector.
In a general torsor, the quantity relating two points
is an element of the acting group.
Thus affine spaces are the commutative linear prototype
of the general theory.

\begin{remark}
The affine-space viewpoint is not merely a matter of terminology.
Its real advantage is conceptual:
it makes clear from the beginning
that the primary objects of interest are not absolute positions,
but differences, displacements, and translations.
This is exactly the structural idea that later reappears
in the general theory of torsors.
\end{remark}

If one chooses a point $a_0\in A$,
then every point $a\in A$ can be written uniquely as
\[
a=a_0+v
\]
for some $v\in V$.
Thus the choice of $a_0$ identifies $A$ with $V$.
But this identification depends on the chosen basepoint,
and there is no canonical reason to prefer one point of $A$
over another.
This is exactly the non-canonical trivialization phenomenon
that is characteristic of torsors in general.

\begin{remark}
A useful advantage of treating the underlying space of analysis
as an affine space is that it prevents one
from regarding the origin as intrinsically distinguished.
In many analytic constructions,
what matters is not a point considered in absolute isolation,
but rather an increment or displacement.
Typical expressions such as
\[
f(x+h)-f(x)
\]
already show this clearly:
the essential quantity is the variation $h$,
not the existence of a preferred origin in the domain.
From this viewpoint,
the domain is naturally an affine space,
while the increments belong to the associated vector space.
This separates points from differences
and makes explicit that analysis is often governed
by relative displacement rather than absolute position.
This is very much in the spirit of the affine viewpoint
in analysis found in Schwartz \cite{Schwartz}.
\end{remark}

\begin{remark}
One may summarize the point as follows:
in an affine approach to analysis,
the point itself belongs to the affine space,
whereas the increment belongs to the associated vector space.
Thus one does not confuse \emph{where one is}
with \emph{how one moves}.
This distinction is conceptually useful already in elementary analysis,
and it becomes even more important in geometry,
where manifolds and configuration spaces are not naturally vector spaces
but are locally governed by displacement data.
For the purposes of these notes,
this is one of the main reasons affine spaces
form the best first model of a torsor.
\end{remark}

\subsection{Solution sets of linear equations}

A second basic example comes from elementary linear algebra.
Very often, the set of solutions to a linear equation
is not naturally a vector space,
but becomes one after a choice of a particular solution.
This is one of the simplest ways in which torsors arise in practice.

\begin{example}
Let $T:V\to W$ be a linear map of vector spaces,
and let $w\in W$.
Consider the solution set
\[
S_w:=\{v\in V\mid T(v)=w\}.
\]
If $S_w$ is nonempty, then it is naturally a torsor
under the additive group of $\ker(T)$.
\end{example}

\begin{proof}
Let $u\in \ker(T)$ and $v\in S_w$.
Then
\[
T(v+u)=T(v)+T(u)=w+0=w,
\]
so $v+u\in S_w$.
Thus $\ker(T)$ acts on $S_w$ by translation.

To see that the action is transitive,
let $v_1,v_2\in S_w$.
Then
\[
T(v_2-v_1)=T(v_2)-T(v_1)=w-w=0,
\]
so $v_2-v_1\in \ker(T)$,
and therefore
\[
v_1+(v_2-v_1)=v_2.
\]
To see that the action is free,
suppose that $v+u=v$ with $u\in\ker(T)$.
Then $u=0$.
Hence the action is free and transitive.
\end{proof}

Thus, whenever a linear equation has at least one solution,
its solution set is not merely a set:
it is an affine space modeled on the kernel.
Equivalently, it is a torsor under $\ker(T)$.
This example is especially useful because it shows
that torsors arise naturally whenever one studies
objects determined only up to homogeneous correction.

\subsection{Cosets}

A third basic class of examples comes from group theory itself.

\begin{example}
Let $G$ be a group and let $H\leq G$ be a subgroup.
For any $g\in G$, the left coset
\[
gH=\{gh\mid h\in H\}
\]
is naturally a right $H$-torsor,
with action given by multiplication on the right:
\[
(gh_1)\cdot h_2:=g(h_1h_2).
\]
Similarly, it may be viewed as a left torsor
after choosing the opposite convention.
\end{example}

\begin{proof}
The right action of $H$ on $gH$ is transitive
because every element of $gH$ has the form $gh$.
It is free because
\[
(gh)\cdot h'=gh
\]
implies
\[
ghh'=gh,
\]
hence $h'=e$.
Therefore $gH$ is a right $H$-torsor.
\end{proof}

This example is conceptually important.
A coset resembles the subgroup $H$,
but does not come with a preferred identity element of its own.
Again, we have a structure that is group-like
without being canonically a group.

\subsection{Sets of bases and frames}

Another instructive example is provided by the set of all ordered bases
of a finite-dimensional vector space.

\begin{example}
Let $V$ be an $n$-dimensional vector space over a field $k$,
and let $\mathcal{B}(V)$ denote the set of ordered bases of $V$.
Then $\mathcal{B}(V)$ is a torsor under the group $\mathrm{GL}(V)$.
\end{example}

\begin{proof}
The group $\mathrm{GL}(V)$ acts on $\mathcal{B}(V)$
by sending a basis $(v_1,\dots,v_n)$ to
\[
(gv_1,\dots,gv_n).
\]
The action is transitive because any ordered basis
can be sent to any other by a unique invertible linear map.
It is free because if $g$ fixes an ordered basis,
then it fixes each basis vector,
hence $g=\mathrm{id}_V$.
Thus the action is free and transitive.
\end{proof}

This example shows that torsors arise naturally
whenever one studies choices of coordinate systems or frames.
A basis is a choice,
but the set of all such choices is often more natural
than any single one of them.
In geometry, this idea becomes even more important:
frame bundles are built from exactly this sort of torsor structure.

\subsection{First comparison of the examples}

Although the preceding examples come from different areas,
they all share the same formal pattern.
In each case:
\begin{enumerate}
\item there is a set of objects or choices,
\item there is a group of allowed transformations,
\item every object can be reached from every other by a unique transformation.
\end{enumerate}
This is precisely the torsor condition.

What varies from example to example
is the interpretation of the acting group.
For affine spaces, it is a vector space acting by translations.
For solution spaces, it is the kernel of a linear map.
For cosets, it is a subgroup acting by multiplication.
For bases, it is the general linear group.
Thus the torsor concept isolates a common structural pattern
underlying apparently different mathematical situations.

\begin{remark}
For the purposes of these notes,
the affine-space example should remain in the foreground.
It is the clearest illustration of the fact
that a torsor has meaningful differences
without possessing a distinguished origin.
Much of the later theory may be understood
as a far-reaching generalization of this simple geometric idea.
\end{remark}

The next section develops this point further
by studying basepoints, differences, and transporters in general.

\section{Exercises on the Basic Theory}

This section contains a small collection of exercises
intended to reinforce the elementary theory of torsors.
Most of them are straightforward,
but they are worth working out carefully,
since the later sections will rely on these ideas.

\begin{enumerate}

\item
Let $G$ act on a set $X$.
Show that the action is free if and only if
the stabilizer of every point is trivial.

\item
Let $G$ act on a set $X$.
Show that the action is transitive if and only if
$X$ consists of a single orbit.

\item
Give an example of an action that is transitive but not free.

\item
Give an example of an action that is free but not transitive.

\item
Let $G$ act freely and transitively on a nonempty set $X$.
Prove directly that for every $x,y\in X$
there exists a unique $g\in G$ such that
\[
g\cdot x=y.
\]

\item
Let $X$ be a $G$-torsor and let $x_0\in X$.
Prove that the map
\[
\phi_{x_0}:G\to X,\qquad g\mapsto g\cdot x_0
\]
is a bijection.

\item
Let $X$ be a $G$-torsor and let $x_0,x_1\in X$.
If $h\in G$ is the unique element such that
\[
h\cdot x_0=x_1,
\]
show that
\[
\phi_{x_1}(g)=\phi_{x_0}(gh)
\]
for all $g\in G$.

\item
Let $X$ be a $G$-torsor.
Prove the transporter identities
\begin{align*}
\operatorname{tr}(x,x)&=e,\\
\operatorname{tr}(x,y)^{-1}&=\operatorname{tr}(y,x),\\
\operatorname{tr}(y,z)\operatorname{tr}(x,y)&=\operatorname{tr}(x,z).
\end{align*}

\item
Let $A$ be an affine space modeled on a vector space $V$.
Show that $A$ is a torsor under the additive group $(V,+)$.

\item
Let $T:V\to W$ be a linear map of vector spaces
and let
\[
S_w=\{v\in V\mid T(v)=w\}.
\]
Assume that $S_w$ is nonempty.
Show that $S_w$ is an affine space modeled on $\ker(T)$.

\item
Let $G$ be a group and $H\leq G$ a subgroup.
Show that each left coset $gH$ is a right $H$-torsor.

\item
Let $V$ be a finite-dimensional vector space.
Show that the set of ordered bases of $V$
is a torsor under $\mathrm{GL}(V)$.

\item
Let $X$ be a $G$-torsor.
Fix a point $x_0\in X$ and define a binary operation on $X$
by transporting the group law along the bijection
\[
G\to X,\qquad g\mapsto g\cdot x_0.
\]
Show that this makes $X$ into a group.
Then show that the resulting group structure depends on the choice of $x_0$.

\item
Explain why the previous exercise does not contradict the statement
that a torsor does not carry a canonical group structure.

\item
Let $X$ be a $G$-torsor.
For fixed $x\in X$, define
\[
d_x(y):=\operatorname{tr}(x,y).
\]
Show that $d_x:X\to G$ is a bijection.
Interpret this as a coordinate system determined by the basepoint $x$.

\item
In your own words, explain the difference between the following two statements:
\begin{enumerate}
\item a torsor becomes identified with a group after a basepoint is chosen;
\item a torsor is canonically a group.
\end{enumerate}

\item
Why is nonemptiness included in the definition of a torsor?
What would go wrong conceptually if the empty set were allowed?

\item
Write a short paragraph explaining why affine spaces
provide the best first model for the general theory of torsors.

\end{enumerate}

\section{Local Triviality and Gluing Data}

Up to this point, torsors have been presented in their elementary form:
a torsor is a nonempty set equipped with a free and transitive group action.
This definition is fundamental,
and nothing in what follows replaces it.
However, if one stops at that point,
one has not yet seen the deeper reason
why torsors arise so naturally in geometry,
descent theory, and local-to-global constructions.

The present section is therefore conceptually important.
Its purpose is to bring to the foreground a second viewpoint,
equally essential for the later theory:
a torsor is not only an object defined by a free and transitive action,
but also an object that is \emph{locally trivial}
and may nevertheless be \emph{globally nontrivial}
because it is assembled from local pieces by gluing data.
In this sense, the absence of a preferred basepoint
is not merely a negative statement.
It is the visible sign of the fact
that local trivializations may exist
without fitting together into a single global one.

\begin{remark}
This shift of viewpoint is one of the main structural points of the paper.
A torsor should not be thought of only as
“a set with a free and transitive action.”
It should also be thought of as
“a locally trivial object whose global form is determined by how local trivializations are glued together.”
The later sheaf-theoretic and topos-theoretic applications
depend heavily on this second perspective.
\end{remark}

This section is meant to isolate that idea at an intuitive level.
The formal sheaf-theoretic and topos-theoretic versions
will be discussed later.
For now, the main point is simply that a torsor,
although not globally endowed with a preferred origin,
may locally admit such choices,
and the change from one local choice to another
is measured by group-valued transition data.

\subsection{Local triviality in the elementary sense}

Let $X$ be a $G$-torsor.
If one chooses a basepoint $x_0\in X$,
then the map
\[
\phi_{x_0}:G\to X,\qquad g\mapsto g\cdot x_0
\]
is a bijection.
Thus, after choosing a point,
the torsor becomes identified with the acting group.
In this sense, a torsor is always \emph{trivialized} by a basepoint.

At first sight, this may seem to reduce a torsor to a copy of the group.
But that is not the correct conclusion.
The crucial point is that the trivialization depends on a choice,
and that choice is generally not canonical.
This is exactly where the gluing viewpoint begins:
one may have trivializations,
but not a preferred one.

\begin{remark}
The importance of local triviality is that it explains
how a torsor can be “group-like” without being canonically a group.
Locally, after a choice of basepoint,
it looks exactly like the acting group.
Globally, however, the lack of a preferred choice remains.
Thus the nonexistence of a canonical origin
is not an accidental omission;
it is the structural reason the torsor must be described by gluing.
\end{remark}

What makes the later theory interesting is that,
in many contexts,
such a choice is available only locally and not globally.
The local triviality of a torsor means that over each sufficiently small region
one can choose a local origin,
or equivalently,
identify the torsor with the acting group.
But these local identifications need not agree on overlaps.
Their disagreement is precisely what gives rise to gluing data.

Even in the elementary set-theoretic situation,
this point is already visible.
A basepoint trivializes the torsor,
but a different basepoint gives a different trivialization.
The passage from one trivialization to another
is controlled by a unique element of the acting group.
The later theory simply organizes this phenomenon systematically
over a family of local regions.

\subsection{From one trivialization to another}

Suppose again that $X$ is a $G$-torsor,
and let $x_0,x_1\in X$.
As seen earlier, there is a unique element $h\in G$ such that
\[
h\cdot x_0=x_1.
\]
The corresponding trivializations
\[
\phi_{x_0},\phi_{x_1}:G\to X
\]
are related by
\[
\phi_{x_1}(g)=\phi_{x_0}(gh).
\]

This formula is elementary,
but it is the prototype for transition functions.
It says that changing the chosen origin
does not destroy the description of the torsor;
rather, it modifies that description by translation in the group.
Thus a torsor does not come with a preferred coordinate system,
but the set of all coordinate systems is itself controlled by the acting group.

\begin{remark}
This is a good place to state the main idea again.
A torsor is locally trivial because each chosen basepoint
identifies it with the acting group.
But the torsor is not globally trivial in any canonical way,
because distinct trivializations differ by nontrivial transition elements.
So the lack of a basepoint should be understood positively:
it is the trace, already in elementary form,
of the gluing phenomenon.
\end{remark}

In geometric situations, this phenomenon occurs repeatedly.
One chooses local trivializations on overlapping regions,
and on each overlap one compares the two choices.
The comparison is encoded by an element, or more generally a function,
with values in the acting group.
These are the transition data of the torsor.

\subsection{The gluing viewpoint}

The gluing viewpoint begins with the following observation.
A torsor need not be given all at once as a globally trivial object.
Instead, it may be constructed from pieces,
each of which is individually trivial,
by specifying how these pieces are to be identified on their overlaps.

To express this informally,
suppose that an ambient space is covered by regions
\[
U_i.
\]
Over each $U_i$,
assume that our torsor looks just like the group $G$ itself.
In other words,
over each region we have chosen a local trivialization.
Then on an overlap $U_i\cap U_j$,
the two local trivializations need not coincide.

Figure~\ref{fig:gluing-two-charts} illustrates the basic gluing situation:
local trivializations are compared on the overlap by transition data.

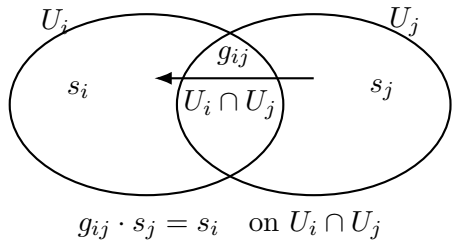
\begin{figure}[H]
\centering
\begin{tikzpicture}[>=Latex, baseline=(current bounding box.center)]
  \draw[thick] (0,0) ellipse (1.8 and 1.2);
  \draw[thick] (2.2,0) ellipse (1.8 and 1.2);

  \node at (-1.2,1.1) {$U_i$};
  \node at (3.4,1.1) {$U_j$};
  \node at (1.1,0) {$U_i\cap U_j$};

  \node at (-0.9,0.2) {$s_i$};
  \node at (3.1,0.2) {$s_j$};

  \draw[->, thick] (2.2,0.35) -- node[above] {$g_{ij}$} (0.1,0.35);

  \node at (1.1,-1.6) {$g_{ij}\cdot s_j = s_i \quad \text{on } U_i\cap U_j$};
\end{tikzpicture}
\caption{Two local trivializations are compared on the overlap
by transition data $g_{ij}$.}
\label{fig:gluing-two-charts}
\end{figure}

Their discrepancy is measured by a transformation
\[
g_{ij}
\]
taking values in $G$.
These transformations describe how to pass
from the $j$-th local picture to the $i$-th local picture.

The essential idea is that the torsor is recovered
not from any single trivialization,
but from the entire compatible family of local trivializations
together with the transition data relating them.
This is precisely the sense in which a torsor is a glued object.

\begin{remark}
Here the conceptual picture should be kept completely explicit:
a torsor is locally trivial,
because over each region it may be identified with the group.
It is globally a glued object,
because these local identifications differ on overlaps.
The absence of a globally distinguished basepoint
is therefore the same structural phenomenon
as the failure of local trivializations to merge into one canonical global trivialization.
\end{remark}

\subsection{Compatibility on overlaps}

Once local trivializations are chosen on a family of regions,
the transition data cannot be arbitrary.
They must satisfy compatibility conditions.

At the most basic level,
if one compares a trivialization with itself,
there should be no change.
Thus one expects
\[
g_{ii}=e.
\]
If one changes from the $i$-th trivialization to the $j$-th
and then changes back,
the result should again be no change.
Thus one expects
\[
g_{ji}=g_{ij}^{-1}.
\]

More importantly,
on a triple overlap
\[
U_i\cap U_j\cap U_k,
\]
one can compare trivializations in two different ways.
One may pass directly from the $k$-th trivialization to the $i$-th,
or one may pass first from $k$ to $j$ and then from $j$ to $i$.
For consistency, these two procedures must agree.
Thus one expects the relation
\[
g_{ij}g_{jk}=g_{ik}.
\]

This is the cocycle condition.
It expresses the simple but crucial fact
that local changes of coordinates must be compatible
if they are to define a coherent global object.

\begin{remark}
The cocycle condition should be read as the algebraic form
of global coherence.
It guarantees that the torsor is not just a collection of local copies of the group,
but a single global object assembled from them.
Without this compatibility,
one would not obtain a well-defined glued object at all.
\end{remark}

\subsection{Why global triviality may fail}

If all local trivializations could be made to agree globally,
then one would obtain a single preferred trivialization of the torsor.
Equivalently, one would obtain a globally chosen basepoint
or a global identification with the acting group.
But in many situations this is impossible.

This failure is exactly what gives torsors much of their interest.
A torsor is locally indistinguishable from the acting group,
yet globally it may remain twisted.
The obstruction is not visible inside any one local piece;
it appears only when one tries to compare all local pieces simultaneously.

\begin{remark}
This is the point that should be retained for the later theory:
the absence of a global origin is not merely an absence.
It is a structural obstruction.
A torsor carries enough local symmetry to look trivial nearby,
but not enough global coherence to force a single global basepoint.
That is why torsors are naturally studied through gluing and cocycles.
\end{remark}

This local/global tension should be kept in mind.
It is one of the main reasons torsors appear naturally
in geometry, cohomology, and descent theory.
The torsor formalism records the fact that
local choices may exist everywhere
without assembling into a single global choice.

\subsection{Affine spaces revisited}

The affine-space example may again be used as a guide.
An affine space is globally trivial once a point is chosen,
but no point is preferred.
Thus it already exhibits, in the simplest possible form,
the principle that local or auxiliary choices
may exist without being canonical.

In more complicated settings,
one may have not merely several possible basepoints,
but several local coordinate systems defined on overlapping domains.
The transitions among these systems are then measured
by the acting group,
just as in the affine case
the choice of one origin rather than another
changes the coordinate description by translation.

Thus affine spaces are not merely the first example of torsors;
they are also the first example of the gluing philosophy behind torsors.

\begin{remark}
Affine spaces make the central slogan especially transparent:
one can work perfectly well with local or chosen origins,
but the geometry itself does not single out one.
The general torsor formalism extends exactly this idea:
local trivializations may exist,
yet the object itself is defined by how those trivializations are related,
not by the presence of a canonical global basepoint.
\end{remark}

\subsection{Preparation for the next section}

The purpose of the present section has been conceptual rather than technical.
We have not yet defined torsors over sheaves,
nor have we formulated the full reconstruction theory.
Instead, we have isolated the basic pattern:

\begin{enumerate}
\item a torsor is locally trivial;
\item different local trivializations are related by group-valued transition data;
\item compatibility of these transitions is expressed by cocycle identities;
\item the failure of a global trivialization is the essential source of twisting.
\end{enumerate}

Let us say the key point one last time.
A torsor is not merely a set with a free and transitive action.
That definition is the formal starting point.
But the deeper geometric meaning of a torsor is that
it is an object which is locally indistinguishable from the acting group
and globally assembled by gluing.
The lack of a preferred basepoint is therefore best understood
as the failure of local trivializations
to combine into a single canonical global trivialization.

The next section makes this more precise
by formulating cocycle descriptions
and explaining how a torsor may be reconstructed from such data.

\section{Cocycle Descriptions and Reconstruction}

In the previous section, we emphasized a change of viewpoint
that is central to the later theory:
a torsor is not only an object defined by a free and transitive action,
but also an object that is locally trivial
and globally assembled by gluing data.
We now make that idea more precise.

The basic principle is simple.
Suppose an object is trivial on each member of a covering.
Then the global object is determined by the way
these local trivial pieces are identified on overlaps.
For torsors, those identifications are measured by elements,
or more generally functions,
with values in the acting group.
The compatibility of these identifications is expressed
by cocycle conditions.

\begin{remark}
This section should be read as the formal counterpart
of the conceptual picture developed in Section~8.
The key idea remains the same:
a torsor is locally a copy of the acting group,
but globally it may be twisted.
What records the twisting is the cocycle data.
\end{remark}

\subsection{Transition data from local trivializations}

Let us begin informally with a family of regions
\[
U_i
\]
covering some ambient space.
Suppose that over each $U_i$,
our torsor has been trivialized.
That is, over each $U_i$,
it has been identified with the acting group $G$.

Now consider an overlap
\[
U_i\cap U_j.
\]
Over this overlap,
we have two trivializations:
the one coming from $U_i$
and the one coming from $U_j$.
Since both describe the same torsor on the overlap,
they must differ by a transformation of the group.
Thus one obtains a transition function
\[
g_{ij}:U_i\cap U_j\to G.
\]

The meaning of $g_{ij}$ is that it tells us
how to pass from the $j$-th local trivialization
to the $i$-th one.
In other words, the family $\{g_{ij}\}$ records
the failure of the local trivializations to agree directly.

\begin{remark}
This is the first precise appearance of the idea
that the lack of a global basepoint
is encoded by the mismatch among local trivializations.
If there were a single global trivialization,
then all local trivializations could be chosen compatibly,
and the transition data would become trivial.
\end{remark}

\subsection{The cocycle identities}

The transition functions $\{g_{ij}\}$ cannot be arbitrary.
They must satisfy the same consistency requirements
discussed informally in Section~8.

First, on each region $U_i$,
comparing a trivialization with itself gives no change,
so one expects
\[
g_{ii}=e.
\]

Second, on each overlap $U_i\cap U_j$,
changing from the $j$-th trivialization to the $i$-th
and then back again should do nothing,
so one expects
\[
g_{ji}=g_{ij}^{-1}.
\]

Third, and most importantly,
on a triple overlap
\[
U_i\cap U_j\cap U_k,
\]
there are two ways to compare the $k$-th trivialization
with the $i$-th:
either directly,
or via the intermediate trivialization on $U_j$.
These must agree.
Hence one obtains
\[
g_{ij}g_{jk}=g_{ik}.
\]

Figure~\ref{fig:cocycle-condition} illustrates the cocycle condition
as the requirement that two different routes of comparison
on a triple overlap give the same result.

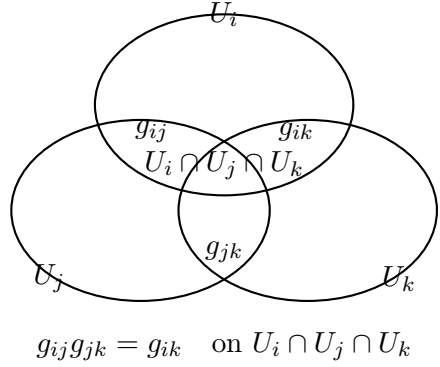
\begin{figure}[H]
\centering
\begin{tikzpicture}[>=Latex, baseline=(current bounding box.center)]
  \draw[thick] (0,0.8) ellipse (1.7 and 1.2);
  \draw[thick] (-1.1,-0.6) ellipse (1.7 and 1.2);
  \draw[thick] (1.1,-0.6) ellipse (1.7 and 1.2);

  \node at (0,2.0) {$U_i$};
  \node at (-2.3,-1.5) {$U_j$};
  \node at (2.3,-1.5) {$U_k$};

  \node at (-0.95,0.45) {$g_{ij}$};
  \node at (0.95,0.45) {$g_{ik}$};
  \node at (0,-1.15) {$g_{jk}$};

  \node at (0,0.0) {$U_i\cap U_j\cap U_k$};

  \node at (0,-2.4) {$g_{ij}g_{jk}=g_{ik}\quad\text{on }U_i\cap U_j\cap U_k$};
\end{tikzpicture}
\caption{On a triple overlap, the two ways of passing between local trivializations
must agree; this is the cocycle condition.}
\label{fig:cocycle-condition}
\end{figure}

This is the cocycle condition.

\begin{definition}
A family of transition functions
\[
g_{ij}:U_i\cap U_j\to G
\]
is called a \emph{$1$-cocycle with values in $G$}
if on all relevant overlaps it satisfies
\begin{align*}
g_{ii}&=e,\\
g_{ji}&=g_{ij}^{-1},\\
g_{ij}g_{jk}&=g_{ik}.
\end{align*}
\end{definition}

\begin{remark}
At this stage, the word ``cocycle'' need not be intimidating.
It simply means a compatible system of transition data.
The cocycle condition is the algebraic formulation
of the fact that local coordinate changes compose consistently.
\end{remark}

\subsection{From a cocycle to a torsor}

Conversely, one may start with cocycle data
and reconstruct a torsor from it.

Suppose we are given local trivial pieces,
each modeled on the group $G$ over a region $U_i$,
together with transition functions
\[
g_{ij}:U_i\cap U_j\to G
\]
satisfying the cocycle condition.
One then glues the local copies of $G$
by declaring that on overlaps,
a point described in the $j$-th trivialization
is identified with the corresponding point
in the $i$-th trivialization
after applying the transition function $g_{ij}$.

The cocycle condition is exactly what guarantees
that this identification process is consistent on triple overlaps.
Without that condition,
one would not obtain a well-defined global object.
With it, one obtains a globally defined torsor
which is locally identified with $G$.

Thus cocycles do not merely record torsors:
they also generate them.

\begin{remark}
This reconstruction viewpoint is crucial.
A torsor is not just an already existing object
that happens to admit local trivializations.
It can actually be built from those trivializations
and the cocycle that compares them.
In this sense, cocycles are not secondary bookkeeping devices;
they are the very gluing instructions of the torsor.
\end{remark}

\subsection{Equivalent cocycles}

Different choices of local trivializations
may produce different cocycles
while still describing the same global torsor.
It is therefore important to understand
when two cocycles should be regarded as equivalent.

Suppose that, on each region $U_i$,
we change the chosen trivialization
by a function
\[
h_i:U_i\to G.
\]
Then the old transition functions $g_{ij}$
are replaced by new transition functions $g'_{ij}$,
related by the formula
\[
g'_{ij}=h_i\,g_{ij}\,h_j^{-1}.
\]

This is simply the algebraic expression of a change of local coordinates.
The global torsor has not changed;
only the local trivializations used to describe it have changed.

\begin{definition}
Two cocycles $\{g_{ij}\}$ and $\{g'_{ij}\}$
are called \emph{equivalent}
if there exist functions $h_i$ on the regions $U_i$
such that
\[
g'_{ij}=h_i\,g_{ij}\,h_j^{-1}
\]
on each overlap $U_i\cap U_j$.
\end{definition}

Equivalent cocycles describe the same torsor.
Thus a torsor is determined not by a single cocycle,
but by an equivalence class of cocycles.

\begin{remark}
This is another place where the central theme should be kept visible.
A torsor is not tied to one preferred trivialization,
and therefore not to one preferred cocycle.
What is intrinsic is the global glued object itself,
not any one local coordinate description of it.
\end{remark}

\subsection{The trivial cocycle and global triviality}

The simplest cocycle is the trivial one,
for which
\[
g_{ij}=e
\]
on every overlap.
In this case, all local trivializations agree directly,
so the local copies of $G$ glue together
without any twisting.
The resulting torsor is globally trivial.

Conversely, if a torsor admits a global trivialization,
then one may choose local trivializations compatible with it,
and the corresponding cocycle becomes trivial.

Thus the trivial cocycle corresponds exactly
to the case in which the torsor has a global origin,
or equivalently,
a global identification with the acting group.

\begin{remark}
This point is worth stating clearly.
The twisting of a torsor is measured by the failure
of its cocycle to be trivial.
So the absence of a global basepoint
is encoded algebraically by nontrivial transition data.
This is one of the cleanest ways to see
that the lack of a canonical origin
is not merely a missing decoration,
but a genuine structural feature.
\end{remark}

\subsection{Why reconstruction matters}

At first sight,
the cocycle description may seem like a technical reformulation
of what was already known.
But it is much more than that.
It is the form in which torsors naturally enter
cohomology, descent theory, and sheaf theory.

What makes the cocycle viewpoint powerful is that it separates
two different ingredients:
\begin{enumerate}
\item local triviality, which makes the object manageable locally;
\item compatibility of transition data, which governs its global structure.
\end{enumerate}
This separation is exactly what is needed
whenever one studies an object that is locally simple
but globally twisted.

For the present notes,
this matters because the later sheaf-theoretic sections
will reinterpret torsors in precisely this local-to-global fashion.
There, local sections and transition data
will play the roles that local trivializations and cocycles
already play here.

\subsection{Summary}

Let us summarize the main content of this section.

\begin{enumerate}
\item A locally trivial torsor determines transition functions on overlaps.
\item These transition functions satisfy cocycle identities.
\item A compatible cocycle can be used to reconstruct a torsor by gluing.
\item Different choices of local trivializations produce equivalent cocycles.
\item A torsor is globally trivial exactly when its cocycle is trivial.
\end{enumerate}

This makes precise the guiding idea of Sections~8 and~9:
a torsor is an object that is locally a copy of the acting group
but globally assembled by gluing.
Its twisting is measured by cocycle data.

The next section extends this picture
from the elementary setting to torsors under sheaves of groups,
where the distinction between local sections and global sections
becomes especially important.

\section{Sheaf Torsors, Local Sections, and Global Sections}

The preceding sections developed the torsor idea
in a deliberately elementary form.
A torsor was first introduced as a set with a free and transitive group action,
and then reinterpreted as a locally trivial object
assembled from compatible gluing data.
We now move to the version that is most relevant
for local-to-global mathematics:
torsors under sheaves of groups.

This step is conceptually important.
In the set-theoretic case,
a torsor is either trivialized by a chosen point or it is not.
In the sheaf-theoretic case,
one may have local sections and local trivializations
without any global section at all.
This distinction between local existence and global existence
is exactly what makes sheaf torsors so useful,
and it is one of the main reasons
they will later be relevant for the intended applications.

\begin{remark}
The central theme of the paper should be kept in mind here.
A torsor is not only a free and transitive action object;
it is also a locally trivial object
whose global structure is governed by gluing.
In the sheaf-theoretic setting,
this local/global contrast becomes especially sharp,
because sections may exist locally without existing globally.
\end{remark}

\subsection{Sheaves of groups}

Let $X$ be a topological space.
A \emph{sheaf of groups} on $X$
is a sheaf $\mathcal{G}$ such that
for every open set $U\subseteq X$,
the set $\mathcal{G}(U)$ is a group,
and the restriction maps are group homomorphisms.

Thus a sheaf of groups assigns to each open set
a group of local symmetries,
in a way that is compatible with restriction to smaller open sets
and with gluing of local sections.
This is the appropriate setting
when the acting symmetry itself varies locally.

\begin{remark}
In the elementary theory,
there was a single global group $G$ acting everywhere.
For sheaf torsors,
the relevant symmetry over an open set $U$
is the group $\mathcal{G}(U)$ of sections of the sheaf on $U$.
So the acting symmetry is now itself local in nature.
\end{remark}

\subsection{Actions of sheaves of groups}

Let $\mathcal{F}$ be a sheaf of sets on $X$,
and let $\mathcal{G}$ be a sheaf of groups on $X$.
A \emph{left action} of $\mathcal{G}$ on $\mathcal{F}$
means that for each open set $U\subseteq X$,
the group $\mathcal{G}(U)$ acts on the set $\mathcal{F}(U)$,
and these actions are compatible with restriction maps.

Concretely, this means that if
\[
g\in \mathcal{G}(U),\qquad s\in \mathcal{F}(U),
\]
then one has an element
\[
g\cdot s\in \mathcal{F}(U),
\]
and if $V\subseteq U$ is open,
then restricting first and acting later
gives the same result as acting first and restricting later.

Thus a sheaf action is simply an action defined uniformly
over all open sets and compatible with localization.

\subsection{Definition of a sheaf torsor}

We now arrive at the sheaf-theoretic analogue of the elementary notion.

\begin{definition}
Let $\mathcal{G}$ be a sheaf of groups on a topological space $X$.
A \emph{$\mathcal{G}$-torsor} is a sheaf of sets $\mathcal{F}$ on $X$
equipped with an action of $\mathcal{G}$
such that:
\begin{enumerate}
\item $\mathcal{F}$ is \emph{locally nonempty}, that is,
there exists an open covering $\{U_i\}$ of $X$
such that $\mathcal{F}(U_i)\neq\varnothing$ for all $i$;
\item for every open set $U\subseteq X$ and every two sections
\[
s,t\in \mathcal{F}(U),
\]
there exists an open covering $\{V_\alpha\}$ of $U$
such that on each $V_\alpha$
there is a unique element
\[
g_\alpha\in \mathcal{G}(V_\alpha)
\]
with
\[
g_\alpha\cdot s|_{V_\alpha}=t|_{V_\alpha}.
\]
\end{enumerate}
\end{definition}

This definition expresses exactly the same torsor idea as before,
but now in a local form.
Sections need not be globally comparable by a single global group element;
instead, they are locally comparable in a unique way.

\begin{remark}
This is the sheaf-theoretic form of the principle
“a torsor is locally trivial.”
The comparison of two sections need not be globally available,
but it must be available locally and uniquely.
So free transitivity is now expressed in a localized way.
\end{remark}

\subsection{Local sections and local triviality}

Let $\mathcal{F}$ be a $\mathcal{G}$-torsor.
By definition, $\mathcal{F}$ is locally nonempty.
So for each point of $X$,
there is some neighborhood $U$
on which one can choose a section
\[
s\in \mathcal{F}(U).
\]

Such a local section plays the role of a local basepoint.
Once it is chosen,
every other local section over the same open set
can be described uniquely by acting on $s$
with a section of the sheaf of groups.

More precisely,
if $s\in \mathcal{F}(U)$ is fixed,
then the map
\[
\mathcal{G}(U)\to \mathcal{F}(U),\qquad g\mapsto g\cdot s,
\]
behaves like a local trivialization:
it identifies the torsor locally with the sheaf of groups.

\begin{remark}
This is the exact sheaf-theoretic analogue
of choosing a basepoint in an ordinary torsor.
A local section trivializes the torsor locally.
But, in general, there need not exist any single global section
that trivializes it everywhere at once.
That is precisely where the gluing problem enters.
\end{remark}

So a sheaf torsor is, by its very definition,
a locally trivial object.
The point is not merely that local sections exist,
but that once such a section is chosen,
the whole torsor over that region becomes indistinguishable
from the sheaf of groups acting on it.

\subsection{Global sections and global triviality}

Figure~\ref{fig:local-vs-global-sections} illustrates the basic local/global distinction:
a sheaf torsor may admit local sections on a covering
without possessing any single global section.

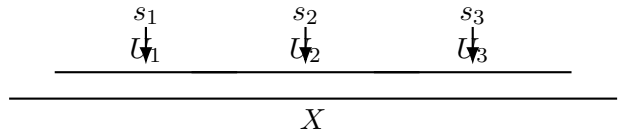
\begin{figure}[H]
\centering
\begin{tikzpicture}[>=Latex, baseline=(current bounding box.center)]
  \draw[thick] (0,0) -- (8,0);
  \node[below] at (4,0) {$X$};

  \draw[thick] (0.6,0.35) -- (3.0,0.35);
  \draw[thick] (2.4,0.35) -- (5.4,0.35);
  \draw[thick] (4.8,0.35) -- (7.4,0.35);

  \node[above] at (1.8,0.35) {$U_1$};
  \node[above] at (3.9,0.35) {$U_2$};
  \node[above] at (6.1,0.35) {$U_3$};

  \node at (1.8,1.1) {$s_1$};
  \node at (3.9,1.1) {$s_2$};
  \node at (6.1,1.1) {$s_3$};

  \draw[->, thick] (1.8,0.95) -- (1.8,0.45);
  \draw[->, thick] (3.9,0.95) -- (3.9,0.45);
  \draw[->, thick] (6.1,0.95) -- (6.1,0.45);

  \node at (4,-1.0) {local sections exist, but no single global section is chosen};
\end{tikzpicture}
\caption{A sheaf torsor may admit local sections on a covering
without admitting a single global section.}
\label{fig:local-vs-global-sections}
\end{figure}

The role of global sections is now clear.
If a sheaf torsor $\mathcal{F}$ admits a global section
\[
s\in \mathcal{F}(X),
\]
then this section serves as a global basepoint.
In that case,
the action of $\mathcal{G}(X)$ on $s$
gives a global trivialization of the torsor.

Thus the existence of a global section means
that the torsor is globally trivial.
Conversely, the failure of a global section
signals that the torsor remains globally twisted,
even though it is locally trivial everywhere.

\begin{remark}
This is one of the most important conceptual points in the paper.
For a sheaf torsor, the difference between local sections and global sections
is the difference between local triviality and global triviality.
The absence of a global section is therefore not merely an omission;
it is the precise mathematical expression
of the failure of local choices to glue into one global choice.
\end{remark}

This is the local/global contrast in its clearest form.
A sheaf torsor may be trivial over every member of a covering
and yet fail to be trivial globally.
Its twisting is therefore not seen locally;
it appears only in the attempt to glue all local trivializations together.

\subsection{Transition data revisited}

Suppose now that $\mathcal{F}$ is a $\mathcal{G}$-torsor
and that we choose local sections
\[
s_i\in \mathcal{F}(U_i)
\]
on an open covering $\{U_i\}$ of $X$.

On an overlap $U_i\cap U_j$,
the sections $s_i$ and $s_j$ both belong to
\[
\mathcal{F}(U_i\cap U_j).
\]
Since $\mathcal{F}$ is a torsor,
there is locally a unique element of $\mathcal{G}$
that transforms one into the other.
Thus one obtains transition data
\[
g_{ij}\in \mathcal{G}(U_i\cap U_j)
\]
such that
\[
g_{ij}\cdot s_j=s_i
\]
on the overlap.

These transition sections satisfy the same cocycle condition
as in the previous section.
Indeed, on triple overlaps,
the uniqueness of transport implies
\[
g_{ij}g_{jk}=g_{ik}.
\]

Thus the whole cocycle picture reappears naturally
inside the sheaf-theoretic setting.
The only difference is that the transition data now live
in a sheaf of groups rather than in one fixed global group.

\begin{remark}
At this point the general philosophy becomes especially transparent.
A sheaf torsor is locally trivial because local sections exist.
It is globally glued because those local sections differ by transition data.
The obstruction to a global section is exactly the failure
of those local trivializations to be globally compatible in a trivial way.
\end{remark}

\subsection{Why local sections matter}

It is worth emphasizing that local sections do more than merely witness
the nonemptiness of the sheaf torsor.
They are the local origins from which the torsor is described.
Without them,
there would be no local trivialization and no transition cocycle.

So in the sheaf-theoretic context,
a section is not just an element of a set.
It is a local choice of reference point.
The existence of such local choices is what makes torsor theory possible,
while the failure of those choices to globalize
is what makes torsor theory interesting.

This observation is one of the main bridges
to the intended applications.
Later, certain protocol-dependent objects
will be interpreted in such a way
that local realizability corresponds to local sections,
while global coherence becomes a more delicate issue.
For that reason,
the reader should become fully comfortable with the idea
that local sections are local basepoints,
whereas global sections are global trivializations.

\subsection{The elementary case as a special case}

It is useful to see that the ordinary torsor theory of sets
fits naturally into this picture.
If one takes the underlying space to be a single point,
then a sheaf of groups is just a group,
a sheaf of sets is just a set,
and a sheaf torsor is exactly an ordinary torsor.

Thus the sheaf-theoretic notion is not a different concept,
but a genuine extension of the elementary one.
The ordinary theory is the special case
in which there is no nontrivial distinction
between local and global data.

\begin{remark}
This is why the sheaf-theoretic version is so natural.
It does not replace the elementary notion of torsor;
it explains what that notion becomes
when one enters a genuinely local-to-global setting.
\end{remark}

\subsection{Summary}

Let us summarize the main points of this section.

\begin{enumerate}
\item A sheaf of groups provides locally varying symmetry.
\item A sheaf torsor is a sheaf of sets with a locally free and transitive action.
\item Local sections play the role of local basepoints.
\item Local sections yield local trivializations.
\item A global section yields a global trivialization.
\item The absence of a global section expresses the failure
of local choices to glue into one global choice.
\item Transition data between local sections satisfy cocycle identities.
\end{enumerate}

This is the form of torsor theory most directly relevant
for later applications.
The next section offers a brief optional glimpse
of how this picture fits naturally into sheaf theory and topos theory
at a more conceptual level.

\section{A Sheaf- and Topos-Theoretic Glimpse}

This section is logically optional for the most elementary reading of the paper,
but it is strongly recommended for readers
who wish to understand later sheaf-theoretic applications in full.
Its purpose is not to develop topos theory systematically,
but to explain why the topos of sheaves provides
the natural ambient universe
for the torsor constructions that motivate these notes.

The main message is simple.
Once one has understood that a torsor is
a locally trivial object assembled from compatible local data,
it becomes natural to ask for a mathematical setting
in which local data, restriction, gluing, and internal symmetry
are all treated uniformly.
Sheaf theory already points in this direction,
and a Grothendieck topos may be regarded
as the natural environment in which this local-to-global logic
is internalized.

\begin{remark}
For the intended later applications,
the topos is not merely a convenient language.
It is the mathematical universe in which local protocol data,
internal symmetries, and global coherence are expressed at the same level.
Thus the phrase ``torsor in a topos''
should be understood not as an abstract embellishment,
but as the natural continuation of the sheaf-theoretic viewpoint.
\end{remark}

\subsection{Why sheaves are not yet the whole story}

In the previous sections,
we repeatedly used the same pattern:
one has local trivializations on a covering,
one compares them on overlaps,
and one asks whether they glue to a global object.
This is already the basic logic of sheaf theory.

A sheaf organizes local data by means of restriction and gluing.
But when one begins to study not only sheaves of sets,
but also internal algebraic structures,
actions, and torsors,
it is useful to pass from an individual sheaf
to the whole category in which such objects live.
That category is a topos.

Thus the topos viewpoint is not a departure from sheaf theory.
It is what one obtains when one decides
to treat sheaf-theoretic local-to-global mathematics
as an ambient universe in its own right.

\begin{remark}
The transition from sheaves to a topos
is analogous to the transition from a single set
to the whole category of sets.
Once one wants to speak not only about particular objects,
but also about groups, actions, quotients, torsors,
and internal constructions among them,
it is natural to work in the whole ambient category.
\end{remark}

\subsection{The topos of sheaves}

Let $(C,J)$ be a site.
The category
\[
\mathbf{Sh}(C,J)
\]
of sheaves on $(C,J)$
is a Grothendieck topos.
For the purposes of these notes,
the reader may think of this as a generalized space of varying local data.
Its objects are sheaves,
its morphisms are morphisms of sheaves,
and its internal logic is adapted to restriction and gluing.

One should keep in mind that
the site $(C,J)$ encodes which families of morphisms are regarded as coverings,
and therefore which notions of locality and compatibility are relevant.
Thus the topos of sheaves is not merely a passive container.
It encodes the local geometry of the problem.

In later applications,
one works not with sheaves on an ordinary topological space,
but with sheaves on a site naturally associated with a protocol.
The resulting topos is then the ambient universe
in which local transcript data, simulation structure,
and torsor-like behavior are organized.

\begin{remark}
This is one of the key reasons for mentioning topoi explicitly.
The topos $\mathbf{Sh}(C,J)$ tells us what ``local''
and ``compatible''
mean for the problem at hand.
So when a later theorem refers to
the topos of sheaves on $(C_\Pi,J_\Pi)$,
the site is not incidental:
it specifies the relevant notion of locality for the protocol $\Pi$.
\end{remark}

Figure~\ref{fig:topos-group-torsor} gives a schematic picture
of the intended ambient setting:
inside the topos $\mathbf{Sh}(C,J)$,
one considers a group object $G$ acting on an object $F$,
and asks whether $F$ admits global sections.

\begin{figure}[H]
\centering
\begin{tikzpicture}[>=Latex, baseline=(current bounding box.center)]
  \draw[thick] (0,0) rectangle (8,4);
  \node[above] at (4,4) {$\mathbf{Sh}(C,J)$};

  \node[draw, rounded corners, minimum width=1.5cm, minimum height=0.8cm] (G) at (2.2,2.4) {$G$};
  \node[draw, rounded corners, minimum width=1.5cm, minimum height=0.8cm] (F) at (5.8,2.0) {$F$};

  \draw[->, thick] (3.0,2.3) -- node[above] {action} (5.0,2.1);

  \node at (2.2,1.4) {group object};
  \node at (5.8,1.0) {torsor object};

  \node[draw, rounded corners, minimum width=2.0cm, minimum height=0.8cm] (GS) at (11.0,2.0) {$\Gamma(F)$};
  \draw[->, thick] (8,2.0) -- node[above] {global sections} (9.2,2.0);
\end{tikzpicture}
\caption{Inside the topos $\mathbf{Sh}(C,J)$,
one may consider a group object $G$, an object $F$ on which $G$ acts,
and the question whether $F$ admits global sections.}
\label{fig:topos-group-torsor}
\end{figure}
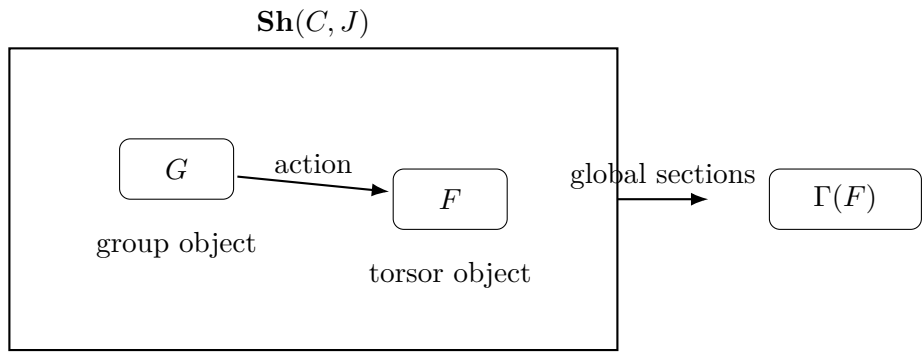

\subsection{Group objects inside a topos}

Inside the ordinary category of sets,
a group is a set with multiplication, identity, and inverse
satisfying the usual axioms.
Inside a topos,
one may define a \emph{group object} in exactly the same formal way:
an object $G$ together with multiplication, identity, and inverse morphisms
satisfying the group axioms internally.

The point is that the definition is the same,
but its meaning is internal to the topos.
Thus one should think of a group object
as a symmetry object living inside the sheaf universe.

In the special case of a sheaf topos,
group objects correspond to sheaves of groups.
So the sheaf-theoretic discussion of the previous section
is already the concrete form of the internal notion.

\begin{remark}
This is an important conceptual simplification.
A group object in a topos is not an exotic new kind of group.
It is the natural internal version of the same idea.
For sheaf topoi, this internal notion matches the familiar external notion
of a sheaf of groups.
\end{remark}

\subsection{Actions and torsors inside a topos}

Once group objects are available,
one may also speak of actions inside the topos.
If $G$ is a group object and $F$ is another object,
an action of $G$ on $F$ is given internally
in exactly the same formal style as in the category of sets.

This leads to the notion of a torsor object.
Informally speaking,
a torsor in a topos is an object
which is locally nonempty
and on which the group object acts
in a locally free and transitive way.
Externally, when the topos is a sheaf topos,
this recovers the notion of a torsor under a sheaf of groups.

So the phrase
``$F$ is a torsor under $G$ in the topos $\mathbf{Sh}(C,J)$''
should be read as follows:
within the universe of sheaves on $(C,J)$,
the object $F$ behaves like a space with symmetry
but without a chosen origin,
and it is locally indistinguishable from the symmetry object $G$.

\begin{remark}
This is one of the main interpretive goals of the present section.
A statement about torsors in a topos
should sound like a natural internalization
of the elementary theory,
not like an unrelated abstract construction.
The formal setting changes,
but the structural idea remains the same:
local triviality, transport, and possible failure of global trivialization.
\end{remark}

\subsection{The global sections functor}

To understand later applications,
one more notion is especially important:
the global sections functor.

If $\mathcal{E}$ is a topos,
one may consider the set of global points or global sections of an object $F$.
In a sheaf topos $\mathbf{Sh}(C,J)$,
this corresponds to taking a section over the terminal object,
that is, a globally defined coherent choice.

Thus, when one says that a torsor object has a global section,
one means that there is a globally coherent basepoint
defined across the whole sheaf universe.
Such a section trivializes the torsor globally.

Conversely, the absence of a global section means
that no single global coherent choice exists,
even though local sections may exist on a covering.
This is the topos-theoretic form
of the local/global distinction emphasized earlier.

\begin{remark}
For the later theorem, this point is crucial.
A global section is not merely ``an element somewhere.''
It is a coherent global choice across the entire site.
The distinction between local sections and global sections
is therefore exactly the distinction between
local realizability and global coherent trivialization.
\end{remark}

\subsection{Internal existence versus local existence}

One of the most useful lessons of the topos viewpoint
is that existence becomes sensitive to locality.
In ordinary set theory,
to say that an object is nonempty
is simply to say that it has an element.
In a sheaf-theoretic or topos-theoretic setting,
an object may be inhabited locally
without admitting any global point.

For torsors this is decisive.
A torsor in a topos may be locally trivial,
because it has local sections on a covering,
while still lacking a global section.
So the torsor is not globally trivial,
even though every local piece looks trivial.

This is exactly the structural phenomenon
that earlier sections described in terms of gluing and cocycles.
The topos viewpoint does not change that phenomenon.
It clarifies it by placing it inside a universe
where local and global existence are formally distinguished.

\begin{remark}
This is one of the deepest conceptual reasons
the topos language is useful here.
It allows one to treat
``there are compatible local realizations''
and
``there is a single global coherent realization''
as mathematically distinct statements.
That distinction will later be essential.
\end{remark}

\subsection{Descent and reconstruction}

The topos-theoretic viewpoint also explains
why torsors belong naturally to the mathematics of descent.

Descent is the principle that
a global object may be reconstructed
from local objects together with compatible gluing data.
But this is exactly how torsors have been described in these notes:
local copies of the acting group,
compared on overlaps by transition data satisfying cocycle conditions.

Thus torsors are among the basic descent objects.
Their cocycle descriptions are descent data,
and their reconstruction from compatible cocycles
is a descent construction.
A topos is the natural setting for this logic,
because it is built to handle local data and gluing systematically.

\begin{remark}
One may therefore summarize the situation as follows:
torsors are natural not only in group theory,
but in the mathematics of descent.
Their local triviality, transition functions, and possible global twisting
are all instances of the general principle
that global structure may be encoded by compatible local data.
\end{remark}

\subsection{Why this is strongly relevant for the later theorem}

We may now say more clearly
why this section is strongly relevant
for the full understanding of the later theorem.

The later theorem does not merely speak about
a set with a group action.
It speaks about:
\begin{enumerate}
\item a site $(C_\Pi,J_\Pi)$ associated with a protocol $\Pi$;
\item the topos $\mathbf{Sh}(C_\Pi,J_\Pi)$ of sheaves on that site;
\item a group object or sheaf of groups $G_\Pi$ in that topos;
\item an object $F_\Pi$ carrying a torsor structure under $G_\Pi$;
\item local triviality and the possible absence of global sections.
\end{enumerate}

Without the present viewpoint,
such a statement may look formally elaborate.
With the present viewpoint,
it should appear as the natural synthesis
of all the themes developed earlier:
locality, gluing, transport, cocycles,
local sections, and global obstruction.

\begin{remark}
This is why the present section,
though logically optional,
is mathematically close to indispensable
for readers who want to understand the later application in depth.
It supplies the ambient interpretation
in which every part of the later theorem becomes conceptually natural.
\end{remark}

\subsection{A heuristic dictionary}

For the reader's convenience,
let us record the main conceptual correspondences once more.

\begin{align*}
\text{ordinary group} &\longrightarrow \text{group object in a topos},\\
\text{ordinary torsor} &\longrightarrow \text{torsor object in a topos},\\
\text{chosen basepoint} &\longrightarrow \text{global section},\\
\text{local basepoint} &\longrightarrow \text{local section},\\
\text{change of trivialization} &\longrightarrow \text{transition cocycle},\\
\text{global triviality} &\longrightarrow \text{existence of a global section},\\
\text{global twisting} &\longrightarrow \text{absence of global section despite local triviality}.
\end{align*}

This dictionary is not a substitute for formal definitions,
but it is the right conceptual guide.

\subsection{Summary}

Let us summarize the main content of this section.

\begin{enumerate}
\item A sheaf topos $\mathbf{Sh}(C,J)$ is the natural universe of local data and gluing.
\item Group objects and torsor objects may be defined internally in that universe.
\item In a sheaf topos, sheaves of groups and sheaf torsors are the concrete external forms of these internal notions.
\item Global sections represent globally coherent choices.
\item Local sections may exist without any global section.
\item This local/global distinction is exactly what makes torsors useful in the later applications.
\item Statements about torsors in a topos should therefore be read
as natural internal versions of the elementary torsor theory developed earlier.
\end{enumerate}

The next section returns to the main line of the paper
and explains more directly
how this torsor language prepares the ground
for later conceptual applications to $\Sigma$-protocols in cryptography.
\section{Toward Applications to $\Sigma$-Protocols}

We now indicate, at a conceptual level,
why the torsor viewpoint developed in the preceding sections
is relevant for later work on $\Sigma$-protocols.
The purpose of the present section is not to give a full formalization
of protocol theory in torsor language.
Rather, it is to explain why such a formalization is natural,
and why the mathematical ideas developed earlier
provide the appropriate preparation for understanding it.

\begin{remark}
The reader should keep in mind that this section is primarily motivational.
Its aim is to show how the language of torsors,
local triviality, cocycles, and local/global sections
can be used to organize certain protocol-theoretic phenomena.
The full technical details belong to later work.
\end{remark}

\subsection{From transport to protocol structure}

At the most elementary level,
a torsor is a space in which relative position is meaningful
even though no preferred origin has been chosen.
This already suggests a useful analogy with protocol-theoretic situations
in which one is given structured data
that can be locally compared, transformed, or transported,
but not necessarily reduced to one globally fixed normal form.

In the case of a $\Sigma$-protocol,
one is concerned with transcripts, witnesses, simulation behavior,
and compatibility conditions among pieces of local data.
The torsor viewpoint suggests that such data
should sometimes be regarded not as isolated algebraic tokens,
but as points in a space organized by transport under a suitable action.
From this perspective,
the relevant mathematical question is not merely
what the underlying algebraic object is,
but whether its points admit local reference choices,
how these choices compare,
and whether they glue globally.

\begin{remark}
This is the first place where the earlier emphasis on transport becomes useful.
What mattered in torsor theory was never only the existence of points,
but the controlled way in which one point is related to another.
That same logic becomes suggestive in protocol settings
where one studies the structured relation among local realizations.
\end{remark}

\subsection{Protocol-dependent sheaves}

The sheaf-theoretic viewpoint suggests the following general pattern.
To a protocol $\Pi$,
one may attempt to associate a sheaf-like object
whose sections over a region of contexts
represent locally realizable transcript data
or locally coherent protocol behavior.
The precise meaning of the “regions” depends on the application:
in later work they are encoded by a site or Grothendieck topology
adapted to the protocol.

What matters conceptually is that one no longer studies
the protocol only through a single global set of transcripts.
Instead, one studies a varying family of local transcript spaces,
together with restriction maps and gluing conditions.
In this way, protocol structure is organized sheaf-theoretically,
and one can begin to ask whether the resulting object
behaves like a torsor under an appropriate sheaf of groups.

\begin{remark}
This is one of the main conceptual bridges to the later theorem.
The shift is from “a protocol has certain transcripts”
to “a protocol determines a sheaf of local transcript data.”
Once that shift has been made,
the torsor language becomes available.
\end{remark}

\subsection{Local triviality and simulation}

One of the most natural correspondences suggested by this framework
concerns local triviality.
In torsor theory, local triviality means that over each sufficiently small region
the object admits a local section,
hence a local basepoint,
and therefore becomes identified with the acting group.

In the protocol-theoretic setting,
the analogous role is played by simulation data.
Very roughly speaking,
a simulator provides, over a suitable local context,
a distinguished way of presenting transcript data
without passing through a single globally fixed witness.
This makes it natural to view simulation
as supplying the analogue of local trivialization.

More explicitly,
one may think of a simulator as producing,
for each admissible local context,
a locally coherent transcript description.
Such a description need not amount to a global witness-based parametrization
of all transcript data at once.
But it does provide a \emph{local reference presentation}
relative to which the relevant structure becomes transparent.
In this sense,
the simulator plays the role of a local trivialization:
it gives a local coordinate system
for the protocol-dependent object,
even when no single global coordinate choice is available.

Figure~\ref{fig:simulation-local-trivialization}
illustrates this analogy schematically.
Each local context $U_i$ carries a local simulated description $s_i$;
these local descriptions play the role of local sections
or local trivializations.
What matters is not that one has already chosen
a single global witness-based origin,
but that one has locally coherent reference data.

\begin{figure}[H]
\centering
\begin{tikzpicture}[>=Latex, baseline=(current bounding box.center)]
  \draw[thick] (0,0) -- (8,0);
  \node[below] at (4,0) {space of local contexts};

  \draw[thick] (0.6,0.35) -- (3.0,0.35);
  \draw[thick] (2.4,0.35) -- (5.4,0.35);
  \draw[thick] (4.8,0.35) -- (7.4,0.35);

  \node[above] at (1.8,0.35) {$U_1$};
  \node[above] at (3.9,0.35) {$U_2$};
  \node[above] at (6.1,0.35) {$U_3$};

  \node at (1.8,1.15) {$s_1$};
  \node at (3.9,1.15) {$s_2$};
  \node at (6.1,1.15) {$s_3$};

  \draw[->, thick] (1.8,1.0) -- (1.8,0.45);
  \draw[->, thick] (3.9,1.0) -- (3.9,0.45);
  \draw[->, thick] (6.1,1.0) -- (6.1,0.45);

  \node at (4,-1.0) {local simulated transcript data act like local trivializations};
\end{tikzpicture}
\caption{Schematic picture of the analogy:
a simulator provides local reference descriptions
that play the role of local trivializations.}
\label{fig:simulation-local-trivialization}
\end{figure}
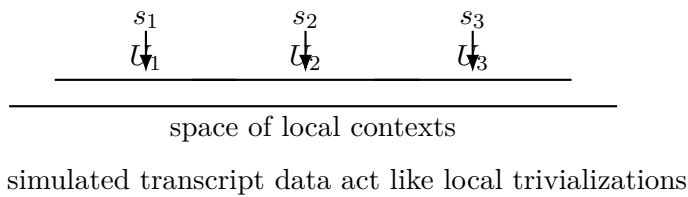

From this point of view,
honest-verifier zero-knowledge suggests local triviality:
the protocol admits locally coherent simulated realizations
which serve as local coordinate systems
for the associated sheaf-like object.
The torsor picture then says that these local coordinate systems
need not arise from a single globally chosen reference point;
they may exist only locally.

\begin{remark}
The analogy should be read structurally rather than literally.
A simulator is not, in any naive set-theoretic sense,
itself simply a section of a sheaf.
Rather, it provides the kind of local reference mechanism
that makes local triviality conceptually plausible.
What the torsor language captures
is the pattern of local reference data,
local transport,
and possible failure of global coherent trivialization.
\end{remark}

\begin{remark}
This is precisely why local sections were emphasized so strongly earlier.
A local section is not merely a local element;
it is a local origin.
In the intended application,
simulation plays an analogous role:
it provides local reference data
relative to which the relevant transcript structure becomes intelligible,
even when no global witness-based trivialization is present.
\end{remark}
\subsection{Why sheaf torsors are the right level}

At this stage, it should be clear
why ordinary set-theoretic torsors would not be enough.
The relevant protocol-theoretic phenomena
are not purely global.
They involve local realizability,
restriction to subcontexts,
compatibility on overlaps,
and possible failure of global coherence.
These are precisely the features for which sheaf torsors are designed.

A sheaf torsor allows one to express:
\begin{enumerate}
\item local existence of sections,
\item unique local transport under a locally varying symmetry,
\item transition data on overlaps,
\item possible failure of global sections.
\end{enumerate}
This is exactly the pattern one expects
when simulation and security are being compared
in a local-to-global framework.

\begin{remark}
The passage from ordinary torsors to sheaf torsors
is therefore not a technical luxury.
It is forced by the structure of the intended application.
Once the protocol is organized sheaf-theoretically,
the natural torsor notion is automatically the sheaf-theoretic one.
\end{remark}

\subsection{Toward torsors in a topos of sheaves}

The optional topos-theoretic section was included
precisely to make the next step feel natural.
If a protocol determines a site $(C_\Pi,J_\Pi)$,
then one may work not merely with isolated sheaves,
but in the entire topos
\[
\mathbf{Sh}(C_\Pi,J_\Pi).
\]
Within that topos,
one may then consider:
\begin{itemize}
\item an object $F_\Pi$ associated with the protocol,
\item a group object or sheaf of groups $G_\Pi$ acting on $F_\Pi$,
\item the possibility that $F_\Pi$ is a torsor under $G_\Pi$.
\end{itemize}

At that point,
a statement such as
“$F_\Pi$ is a torsor under $G_\Pi$ in the topos of sheaves on $(C_\Pi,J_\Pi)$”
should no longer sound abrupt.
It is simply the sheaf-theoretic and topos-theoretic continuation
of the elementary theory developed in this note:
local triviality, transport, gluing, and possible global obstruction.

\begin{remark}
This is one of the main pedagogical goals of the paper.
A reader who has followed the earlier sections
should now be able to hear the phrase
“torsor in the topos of sheaves”
as a natural generalization of
“group action without a chosen origin,”
rather than as an unrelated abstract construction.
\end{remark}

\subsection{A heuristic dictionary}

For later use,
it is helpful to summarize the conceptual correspondences
suggested by the torsor viewpoint:

\begin{center}
\begin{tabular}{ll}
ordinary torsor theory & protocol-theoretic heuristic \\[4pt]
chosen basepoint & chosen local reference data \\[4pt]
local section & local realizability / local simulation datum \\[4pt]
global section & global coherent trivialization \\[4pt]
transition cocycle & compatibility data among local descriptions \\[4pt]
nontrivial torsor & globally twisted but locally trivial structure
\end{tabular}
\end{center}

This dictionary is intentionally heuristic,
but it captures the reason the torsor formalism is useful.
It organizes the relation between
local realizability and global coherence
in a way that ordinary pointwise language does not.

\subsection{Preparation for the later theorem}

The later theorem that motivates these notes
may be read schematically as follows:
under suitable hypotheses,
a protocol-dependent sheaf-like object $F_\Pi$
carries a torsor structure under a corresponding symmetry object $G_\Pi$.
Its local triviality is related to simulation behavior,
while the failure of suitable global sections
reflects structural security constraints.

The present note does not prove that theorem.
What it aims to do is make such a theorem intelligible.
After reading the present paper,
the reader should understand why each part of that statement is natural:
\begin{enumerate}
\item why one associates sheaf-like objects to local protocol data;
\item why local transport suggests torsor structure;
\item why local triviality corresponds to simulation-like behavior;
\item why the existence or nonexistence of global sections matters.
\end{enumerate}

If this conceptual preparation has been successful,
then the later theorem will not look like an arbitrary formal trick.
It will appear instead as a natural synthesis
of the basic mathematics of torsors,
local triviality, cocycles, sheaves, and topos-theoretic gluing.

\subsection{Summary}

Let us summarize the purpose of the present section.

\begin{enumerate}
\item Torsor language is useful for organizing protocol-dependent local data.
\item A sheaf-theoretic protocol object may admit local trivializations
suggested by simulation.
\item Global sections correspond to globally coherent trivializations.
\item The failure of such global coherence
may reflect structural security constraints.
\item This makes torsors in sheaf topoi
a natural language for later work on $\Sigma$-protocols.
\end{enumerate}

The next section concludes the paper
by summarizing the mathematical path we have followed:
from group actions and affine spaces
to gluing, sheaf torsors, and the conceptual preparation
for later applications in cryptography.

\section{Conclusion}

The main purpose of these notes has been to provide
a preparatory introduction to torsors
with a view toward later applications to $\Sigma$-protocols in cryptography.
Rather than attempting a comprehensive survey
of all appearances of torsors in mathematics,
we have focused on a more specific conceptual path:
from group actions and affine spaces
to local triviality, gluing data, sheaf torsors,
and the contrast between local and global sections.

The elementary starting point was simple.
A torsor is a nonempty set equipped with a free and transitive group action.
Equivalently, it is a space in which every point can be transported
to every other point by a unique group element.
From this perspective,
torsors may be viewed as spaces with symmetry
but without a chosen origin.
This already explains why affine spaces provide the basic model:
differences are meaningful,
while absolute positions require an auxiliary choice.

We then emphasized that this elementary definition,
though fundamental, does not exhaust the meaning of the concept.
A torsor is not merely an object defined by a free and transitive action.
It is also an object that is locally trivial
and may be globally nontrivial because it is assembled from local pieces
by gluing data.
This second viewpoint is essential.
It explains why torsors appear naturally in geometry,
descent, and local-to-global mathematics,
and it shows that the lack of a preferred basepoint
should be understood positively:
it reflects the failure of local trivializations
to combine into one canonical global trivialization.

From there,
the cocycle description made the gluing picture more precise.
A torsor determines transition data on overlaps,
these data satisfy cocycle identities,
and conversely a compatible cocycle may be used
to reconstruct the torsor.
In this way,
torsors become among the simplest and clearest examples
of global objects encoded by local comparison data.

The sheaf-theoretic sections then extended this picture
to the context most relevant for later work.
A torsor under a sheaf of groups is locally trivial by local sections,
but need not admit any global section.
Thus the distinction between local existence and global existence,
already implicit in the gluing viewpoint,
becomes mathematically explicit.
The optional topos-theoretic section was included
to indicate that this is not an accidental feature of sheaves alone,
but part of a more general local-to-global logic
internal to a topos.

Finally, we sketched why this language is useful
for thinking about $\Sigma$-protocols.
The point was not to replace the usual algebraic theory of such protocols,
but to prepare a conceptual framework
in which protocol-dependent objects may be organized sheaf-theoretically,
local triviality may be related to simulation-like behavior,
and the failure of appropriate global sections
may reflect structural security constraints.
In that setting,
the statement that a protocol-dependent object is a torsor
in a topos of sheaves should no longer seem abrupt;
it is the natural continuation
of the mathematics developed in these notes.

If the exposition has achieved its aim,
the reader should now see torsors in two complementary ways.
On the one hand,
a torsor is a space of unique transport under a group action.
On the other hand,
it is a locally trivial object whose global form is determined by gluing.
Together, these two viewpoints form the mathematical background
for the later sheaf-theoretic and cryptographic applications
that motivate the present work.

\bigskip

\noindent Takao Inou\'{e}

\noindent Faculty of Informatics

\noindent Yamato University

\noindent Katayama-cho 2-5-1, Suita, Osaka, 564-0082, Japan

\noindent inoue.takao@yamato-u.ac.jp
 
\noindent (Personal) takaoapple@gmail.com (I prefer my personal mail)
\bigskip

\appendix

\section{Solutions to the Exercises in Section 7}

This appendix contains detailed solutions to the exercises in Section~7.
The purpose is not merely to record correct answers,
but to illustrate how the basic definitions are actually used.
In particular, the reader should pay attention to the repeated role of
the definitions of free action, transitive action, torsor,
and transporter.
These notions are simple, but they govern everything that follows.

\begin{enumerate}

\item
\textbf{Exercise.}
Let $G$ act on a set $X$.
Show that the action is free if and only if
the stabilizer of every point is trivial.

\medskip

\textbf{Solution.}
Recall the definitions.

The action is \emph{free} if, for every $x\in X$ and every $g\in G$,
\[
g\cdot x=x \Longrightarrow g=e.
\]
Also, the stabilizer of a point $x\in X$ is
\[
G_x:=\{g\in G\mid g\cdot x=x\}.
\]

Assume first that the action is free.
Fix $x\in X$.
If $g\in G_x$, then by definition $g\cdot x=x$.
Since the action is free, this implies $g=e$.
Thus $G_x=\{e\}$.

Conversely, assume that $G_x=\{e\}$ for every $x\in X$.
Suppose that $g\cdot x=x$ for some $g\in G$ and $x\in X$.
Then $g\in G_x$, so $g=e$.
Hence the action is free.

So the action is free exactly when every stabilizer is trivial.

\medskip

\textbf{Comment.}
This is a basic example of unpacking two definitions
and checking that they say the same thing in different language.
In practice, freeness is often verified by computing stabilizers.

\item
\textbf{Exercise.}
Let $G$ act on a set $X$.
Show that the action is transitive if and only if
$X$ consists of a single orbit.

\medskip

\textbf{Solution.}
Recall that the action is \emph{transitive} if,
for every $x,y\in X$, there exists $g\in G$ such that
\[
g\cdot x=y.
\]
Also, the orbit of $x$ is
\[
G\cdot x=\{g\cdot x\mid g\in G\}.
\]

Assume first that the action is transitive.
Fix any $x\in X$.
For every $y\in X$, transitivity gives some $g\in G$ such that
\[
g\cdot x=y.
\]
Hence $y\in G\cdot x$.
So every point of $X$ lies in the orbit of $x$, and therefore
\[
X=G\cdot x.
\]
Thus $X$ consists of a single orbit.

Conversely, suppose that $X$ consists of a single orbit.
Then there exists some $x\in X$ such that
\[
X=G\cdot x.
\]
Now let $y,z\in X$.
Since $y$ and $z$ lie in the same orbit, there are elements $g,h\in G$ such that
\[
y=g\cdot x,\qquad z=h\cdot x.
\]
Then
\[
z=h\cdot x=h g^{-1}\cdot (g\cdot x)=h g^{-1}\cdot y.
\]
So $z$ is obtained from $y$ by the action of an element of $G$.
Hence the action is transitive.

\medskip

\textbf{Comment.}
Again, this is a matter of translating between two forms of the same idea:
transitivity means that all points communicate with one another,
and “one orbit” expresses exactly that.

\item
\textbf{Exercise.}
Give an example of an action that is transitive but not free.

\medskip

\textbf{Solution.}
Let a group $G$ act on the set of left cosets $G/H$,
where $H\leq G$ is a nontrivial subgroup.
The action is given by left multiplication:
\[
g\cdot (aH):=(ga)H.
\]

This action is transitive:
given any two cosets $aH$ and $bH$, the element $ba^{-1}\in G$ satisfies
\[
(ba^{-1})\cdot (aH)=bH.
\]

However, the action is not free if $H\neq \{e\}$.
Indeed, the stabilizer of the coset $H$ is exactly $H$ itself.
Since $H$ is nontrivial, the stabilizer is nontrivial,
so the action is not free.

\medskip

\textbf{Comment.}
This example is important because coset actions are one of the standard ways
to see transitivity without freeness.

\item
\textbf{Exercise.}
Give an example of an action that is free but not transitive.

\medskip

\textbf{Solution.}
Let $G=\mathbb{Z}$ act on $X=\mathbb{Z}\sqcup \mathbb{Z}$,
the disjoint union of two copies of $\mathbb{Z}$,
by translation in each copy:
\[
n\cdot (m,\varepsilon):=(n+m,\varepsilon),
\]
where $\varepsilon\in\{0,1\}$ indicates which copy we are in.

This action is free:
if
\[
n\cdot (m,\varepsilon)=(m,\varepsilon),
\]
then
\[
(n+m,\varepsilon)=(m,\varepsilon),
\]
so $n=0$.

But it is not transitive,
because points in one copy of $\mathbb{Z}$ cannot be moved
into the other copy.
Thus there are at least two distinct orbits.

\medskip

\textbf{Comment.}
This example shows that freeness is a local condition at each point,
whereas transitivity is a global condition on how the whole set hangs together.

\item
\textbf{Exercise.}
Let $G$ act freely and transitively on a nonempty set $X$.
Prove directly that for every $x,y\in X$
there exists a unique $g\in G$ such that
\[
g\cdot x=y.
\]

\medskip

\textbf{Solution.}
We use both parts of the assumption.

Because the action is transitive,
for any given $x,y\in X$ there exists at least one $g\in G$ such that
\[
g\cdot x=y.
\]

To prove uniqueness, suppose that both $g,h\in G$ satisfy
\[
g\cdot x=y,\qquad h\cdot x=y.
\]
Then
\[
h^{-1}g\cdot x=h^{-1}\cdot (g\cdot x)=h^{-1}\cdot y=x.
\]
Since the action is free, the only element fixing $x$ is the identity.
Therefore
\[
h^{-1}g=e,
\]
hence $g=h$.

So there exists a unique such $g$.

\medskip

\textbf{Comment.}
This exercise is the key calculation behind torsor theory.
Transitivity gives existence of transport,
and freeness gives uniqueness of transport.

\item
\textbf{Exercise.}
Let $X$ be a $G$-torsor and let $x_0\in X$.
Prove that the map
\[
\phi_{x_0}:G\to X,\qquad g\mapsto g\cdot x_0
\]
is a bijection.

\medskip

\textbf{Solution.}
Since $X$ is a torsor, the action is free and transitive.

To prove surjectivity,
let $x\in X$.
By transitivity, there exists $g\in G$ such that
\[
g\cdot x_0=x.
\]
So every point of $X$ lies in the image of $\phi_{x_0}$.

To prove injectivity,
suppose
\[
\phi_{x_0}(g)=\phi_{x_0}(h).
\]
Then
\[
g\cdot x_0=h\cdot x_0.
\]
Applying $h^{-1}$ to both sides,
\[
h^{-1}g\cdot x_0=x_0.
\]
Since the action is free, $h^{-1}g=e$, so $g=h$.

Therefore $\phi_{x_0}$ is a bijection.

\medskip

\textbf{Comment.}
This is the precise mathematical content of the slogan
“a torsor becomes a copy of the group once a basepoint is chosen.”

\item
\textbf{Exercise.}
Let $X$ be a $G$-torsor and let $x_0,x_1\in X$.
If $h\in G$ is the unique element such that
\[
h\cdot x_0=x_1,
\]
show that
\[
\phi_{x_1}(g)=\phi_{x_0}(gh)
\]
for all $g\in G$.

\medskip

\textbf{Solution.}
By definition,
\[
\phi_{x_1}(g)=g\cdot x_1.
\]
Since $x_1=h\cdot x_0$, this becomes
\[
\phi_{x_1}(g)=g\cdot (h\cdot x_0)=(gh)\cdot x_0.
\]
But the right-hand side is exactly
\[
\phi_{x_0}(gh).
\]
So
\[
\phi_{x_1}(g)=\phi_{x_0}(gh)
\]
for all $g\in G$.

\medskip

\textbf{Comment.}
This shows explicitly how changing the basepoint changes the coordinate description:
the two identifications differ by translation in the group.

\item
\textbf{Exercise.}
Let $X$ be a $G$-torsor.
Prove the transporter identities
\begin{align*}
\operatorname{tr}(x,x)&=e,\\
\operatorname{tr}(x,y)^{-1}&=\operatorname{tr}(y,x),\\
\operatorname{tr}(y,z)\operatorname{tr}(x,y)&=\operatorname{tr}(x,z).
\end{align*}

\medskip

\textbf{Solution.}
By definition, $\operatorname{tr}(x,y)$ is the unique element of $G$ such that
\[
\operatorname{tr}(x,y)\cdot x=y.
\]

For the first identity, both $e$ and $\operatorname{tr}(x,x)$ send $x$ to $x$.
By uniqueness,
\[
\operatorname{tr}(x,x)=e.
\]

For the second identity, let
\[
g=\operatorname{tr}(x,y).
\]
Then $g\cdot x=y$, so
\[
x=g^{-1}\cdot y.
\]
Thus $g^{-1}$ sends $y$ to $x$.
By uniqueness,
\[
g^{-1}=\operatorname{tr}(y,x).
\]
So
\[
\operatorname{tr}(x,y)^{-1}=\operatorname{tr}(y,x).
\]

For the third identity, let
\[
g=\operatorname{tr}(x,y),\qquad h=\operatorname{tr}(y,z).
\]
Then
\[
g\cdot x=y,\qquad h\cdot y=z.
\]
Hence
\[
(hg)\cdot x=h\cdot(g\cdot x)=h\cdot y=z.
\]
So $hg$ sends $x$ to $z$.
By uniqueness,
\[
hg=\operatorname{tr}(x,z).
\]
That is,
\[
\operatorname{tr}(y,z)\operatorname{tr}(x,y)=\operatorname{tr}(x,z).
\]

\medskip

\textbf{Comment.}
These identities show that transporters behave like generalized differences:
identity means no displacement, inverse means reversing direction,
and multiplication means composing successive displacements.

\item
\textbf{Exercise.}
Let $A$ be an affine space modeled on a vector space $V$.
Show that $A$ is a torsor under the additive group $(V,+)$.

\medskip

\textbf{Solution.}
By definition of affine space, there is an action
\[
A\times V\to A,\qquad (a,v)\mapsto a+v,
\]
satisfying:
\begin{enumerate}
\item $a+0=a$,
\item $(a+v)+w=a+(v+w)$,
\item for every $a,b\in A$, there exists a unique $v\in V$ such that
\[
a+v=b.
\]
\end{enumerate}

The first two properties say that $(V,+)$ acts on $A$.
The third says exactly that for every $a,b\in A$,
there exists a unique $v\in V$ sending $a$ to $b$.
That is precisely the condition that the action be free and transitive.

So $A$ is a torsor under the additive group $(V,+)$.

\medskip

\textbf{Comment.}
This is the most important example in the whole note.
It is the basic geometric model for all later torsor ideas.

\item
\textbf{Exercise.}
Let $T:V\to W$ be a linear map of vector spaces
and let
\[
S_w=\{v\in V\mid T(v)=w\}.
\]
Assume that $S_w$ is nonempty.
Show that $S_w$ is an affine space modeled on $\ker(T)$.

\medskip

\textbf{Solution.}
We define an action of $\ker(T)$ on $S_w$ by translation:
for $u\in\ker(T)$ and $v\in S_w$, set
\[
v+u.
\]

First we check that this stays inside $S_w$.
Since $v\in S_w$, we have $T(v)=w$.
Since $u\in\ker(T)$, we have $T(u)=0$.
Therefore
\[
T(v+u)=T(v)+T(u)=w+0=w.
\]
So $v+u\in S_w$.

Now we verify the affine-space axioms.
Clearly
\[
v+0=v,\qquad (v+u_1)+u_2=v+(u_1+u_2).
\]
Finally, let $v_1,v_2\in S_w$.
Then
\[
T(v_2-v_1)=T(v_2)-T(v_1)=w-w=0,
\]
so $v_2-v_1\in\ker(T)$.
Also,
\[
v_1+(v_2-v_1)=v_2.
\]
This element of $\ker(T)$ is unique,
because if
\[
v_1+u=v_2,
\]
then necessarily
\[
u=v_2-v_1.
\]

Thus $S_w$ is an affine space modeled on $\ker(T)$.

\medskip

\textbf{Comment.}
This exercise is a good example of the general principle
“solutions differ by homogeneous solutions.”

\item
\textbf{Exercise.}
Let $G$ be a group and $H\leq G$ a subgroup.
Show that each left coset $gH$ is a right $H$-torsor.

\medskip

\textbf{Solution.}
Define a right action of $H$ on $gH$ by
\[
(gh_1)\cdot h_2:=g(h_1h_2).
\]

First, this is well defined because every element of $gH$ has the form $gh_1$,
and multiplying on the right by an element of $H$ stays inside $gH$.

The action axioms are straightforward:
\[
(gh)\cdot e=gh,\qquad ((gh)\cdot h_1)\cdot h_2=(gh)\cdot (h_1h_2).
\]

To prove transitivity,
take any two elements $gh_1,gh_2\in gH$.
Then
\[
(gh_1)\cdot (h_1^{-1}h_2)=gh_2.
\]
So any point can be moved to any other.

To prove freeness,
suppose
\[
(gh)\cdot h'=gh.
\]
Then
\[
ghh'=gh.
\]
Multiplying on the left by $(gh)^{-1}$ gives
\[
h'=e.
\]
So the action is free.

Hence $gH$ is a right $H$-torsor.

\medskip

\textbf{Comment.}
A coset looks like the subgroup, but it has no preferred identity element.
That is why it is a torsor rather than canonically a group.

\item
\textbf{Exercise.}
Let $V$ be a finite-dimensional vector space.
Show that the set of ordered bases of $V$
is a torsor under $\mathrm{GL}(V)$.

\medskip

\textbf{Solution.}
Let $\mathcal{B}(V)$ be the set of ordered bases of $V$.
Define an action of $\mathrm{GL}(V)$ on $\mathcal{B}(V)$ by
\[
g\cdot (v_1,\dots,v_n):=(gv_1,\dots,gv_n).
\]

This is indeed an action, since the identity map fixes each basis,
and composition of linear maps corresponds to composition of actions.

To prove transitivity,
let
\[
(v_1,\dots,v_n),\qquad (w_1,\dots,w_n)
\]
be two ordered bases.
There is a unique linear automorphism $g\in\mathrm{GL}(V)$ such that
\[
g(v_i)=w_i
\]
for all $i$.
Hence
\[
g\cdot (v_1,\dots,v_n)=(w_1,\dots,w_n).
\]

To prove freeness,
suppose
\[
g\cdot (v_1,\dots,v_n)=(v_1,\dots,v_n).
\]
Then $g(v_i)=v_i$ for every basis vector $v_i$.
Since a linear map is determined by its values on a basis,
we conclude that $g=\mathrm{id}_V$.

So the action is free and transitive,
hence $\mathcal{B}(V)$ is a $\mathrm{GL}(V)$-torsor.

\medskip

\textbf{Comment.}
This is the prototype for frame bundles in geometry.

\item
\textbf{Exercise.}
Let $X$ be a $G$-torsor.
Fix a point $x_0\in X$ and define a binary operation on $X$
by transporting the group law along the bijection
\[
G\to X,\qquad g\mapsto g\cdot x_0.
\]
Show that this makes $X$ into a group.
Then show that the resulting group structure depends on the choice of $x_0$.

\medskip

\textbf{Solution.}
Since $X$ is a $G$-torsor, the map
\[
\phi_{x_0}:G\to X,\qquad g\mapsto g\cdot x_0
\]
is a bijection.

Define a multiplication on $X$ by
\[
x*y:=\phi_{x_0}\bigl(\phi_{x_0}^{-1}(x)\,\phi_{x_0}^{-1}(y)\bigr).
\]
This is the standard transport of structure along a bijection.

Because $\phi_{x_0}$ is a bijection and $G$ is a group,
it follows immediately that $(X,*)$ is a group:
associativity, identity, and inverses are all carried over from $G$.

The identity element of this group structure is
\[
\phi_{x_0}(e)=x_0.
\]
So the chosen basepoint becomes the identity element.

Now choose another basepoint $x_1\in X$.
Repeating the same construction using $x_1$
produces another group law on $X$ whose identity element is $x_1$.
Unless $x_1=x_0$, the two identity elements are different,
so the two group structures are different.

Therefore the group structure depends on the choice of basepoint.

\medskip

\textbf{Comment.}
This is one of the most important conceptual exercises in the note.
It shows that a torsor can be made into a group,
but only after a noncanonical choice.

\item
\textbf{Exercise.}
Explain why the previous exercise does not contradict the statement
that a torsor does not carry a canonical group structure.

\medskip

\textbf{Solution.}
The previous exercise shows only that
\emph{after choosing a basepoint}
one can transport the group structure of $G$ onto $X$.
But the resulting group law depends on that choice.

A canonical group structure would be one determined intrinsically by $X$ itself,
without any extra choice.
That is exactly what a torsor lacks.
A torsor has enough structure to compare points by transport,
but not enough structure to single out one point
as the identity element in a preferred way.

So there is no contradiction:
a torsor may admit many group structures,
one for each chosen basepoint,
without carrying a canonical one.

\item
\textbf{Exercise.}
Let $X$ be a $G$-torsor.
For fixed $x\in X$, define
\[
d_x(y):=\operatorname{tr}(x,y).
\]
Show that $d_x:X\to G$ is a bijection.
Interpret this as a coordinate system determined by the basepoint $x$.

\medskip

\textbf{Solution.}
Fix $x\in X$.
For each $y\in X$, the transporter $\operatorname{tr}(x,y)$ is the unique element of $G$
sending $x$ to $y$.
So the map
\[
d_x:X\to G,\qquad y\mapsto \operatorname{tr}(x,y)
\]
assigns to each point its displacement from the basepoint $x$.

To show surjectivity,
let $g\in G$.
Since the action is defined on $X$, the point
\[
y:=g\cdot x
\]
lies in $X$.
By definition of transporter,
\[
\operatorname{tr}(x,y)=g.
\]
So $d_x$ is surjective.

To show injectivity,
suppose
\[
d_x(y_1)=d_x(y_2).
\]
Then
\[
\operatorname{tr}(x,y_1)=\operatorname{tr}(x,y_2).
\]
Applying both sides to $x$,
we get
\[
y_1=y_2.
\]
So $d_x$ is injective.

Thus $d_x$ is a bijection.

This is a coordinate system determined by the basepoint $x$:
every point is described by its unique displacement from $x$.

\medskip

\textbf{Comment.}
This is the transporter version of the map $\phi_x:G\to X$.
One goes from group elements to points, the other from points to group elements.

\item
\textbf{Exercise.}
In your own words, explain the difference between the following two statements:
\begin{enumerate}
\item a torsor becomes identified with a group after a basepoint is chosen;
\item a torsor is canonically a group.
\end{enumerate}

\medskip

\textbf{Solution.}
The first statement means that,
once one chooses a point of the torsor,
one can use that point as an origin
and thereby identify the torsor with the acting group.
This identification is perfectly valid,
but it depends on the chosen point.

The second statement would mean that
the torsor carries a group structure intrinsically,
without any extra choice,
and that one point is distinguished as the identity in a natural way.

The difference is therefore the difference between
\emph{noncanonical identification after a choice}
and
\emph{intrinsic structure without a choice}.
A torsor has the first property but not, in general, the second.

\item
\textbf{Exercise.}
Why is nonemptiness included in the definition of a torsor?
What would go wrong conceptually if the empty set were allowed?

\medskip

\textbf{Solution.}
A torsor is meant to be a space in which
every pair of points can be compared by unique transport.
If there are no points at all,
this basic intuition disappears.

Formally, if the empty set were allowed,
some conditions could hold vacuously,
but the resulting object would no longer deserve to be called a torsor
in the intended sense.
There would be no basepoints, no transporters, no local coordinates,
and no possibility of comparing points.

This matters even more later in sheaf theory,
where one distinguishes carefully between
local existence of sections and global existence of sections.
The ordinary definition of torsor starts with a genuinely inhabited space.

\item
\textbf{Exercise.}
Write a short paragraph explaining why affine spaces
provide the best first model for the general theory of torsors.

\medskip

\textbf{Solution.}
Affine spaces provide the best first model for torsors
because they express the central idea in the most familiar setting:
one can compare points by differences,
but there is no canonically distinguished origin.
The acting group is simply the additive group of the underlying vector space,
so the torsor structure is easy to visualize as translation.
At the same time, affine spaces already display
the key noncanonical feature of torsors:
choosing a point identifies the affine space with the vector space,
but no such choice is intrinsic.
Thus affine spaces give a concrete geometric prototype
for the general notion of symmetry without a chosen origin.

\end{enumerate}

\section{A Suggested Lecture Plan}

This appendix proposes one possible way
to use the present notes for a short lecture course.
The plan is designed for classes of approximately $90$ minutes each.
It is only a suggestion:
depending on the audience,
one may compress the material into fewer meetings
or expand it into a slower course with more exercises.

A natural standard format is a course of nine lectures.
The first half develops the elementary theory of torsors,
while the second half moves toward local triviality,
gluing, sheaf torsors, and the conceptual bridge
to later applications involving $\Sigma$-protocols.

\subsection{A nine-lecture plan}

\begin{enumerate}
\item \textbf{Introduction; Group Actions and Orbits.}\\
Introduce the general aim of the notes.
Review group actions, orbits, and stabilizers,
with emphasis on the idea of transport between points.

\item \textbf{Free and Transitive Actions; Definition of a Torsor.}\\
Explain why free and transitive actions are the key structure.
Introduce torsors formally and discuss the slogan
“a group without a chosen identity.”

\item \textbf{Basic Examples I: Affine Spaces.}\\
Develop affine spaces as the guiding example.
Emphasize the distinction between points and differences,
and explain why this viewpoint is useful in geometry and analysis.

\item \textbf{Basic Examples II; Basepoints, Differences, and Transporters.}\\
Discuss solution sets of linear equations, cosets,
and sets of bases.
Then explain basepoints, non-canonical identifications,
and transporter identities.

\item \textbf{Exercises on the Basic Theory.}\\
Work through selected exercises from Section~7.
At this stage it is especially useful
to emphasize how the definitions are used in practice.
The solutions in the appendix may be used as supporting material.

\item \textbf{Local Triviality and Gluing Data.}\\
Introduce the second major viewpoint on torsors:
they are locally trivial objects
that may be globally nontrivial.
Explain transition data and the basic gluing philosophy.

\item \textbf{Cocycle Descriptions and Reconstruction.}\\
Develop transition functions, cocycle identities,
equivalence of cocycles,
and the reconstruction of torsors from cocycle data.

\item \textbf{Sheaf Torsors, Local Sections, and Global Sections.}\\
Move to torsors under sheaves of groups.
Emphasize the distinction between local sections and global sections,
and explain why global triviality is stronger than local triviality.

\item \textbf{Topos-Theoretic Glimpse; Toward $\Sigma$-Protocols.}\\
Present the optional topos-theoretic perspective,
especially for stronger students.
Then explain how the torsor viewpoint prepares the ground
for later applications to $\Sigma$-protocols in cryptography.
Conclude with a summary of the whole course.
\end{enumerate}

\subsection{Possible variations}

For a shorter course,
Lectures 3 and 4 may be combined,
and the optional topos-theoretic material may be reduced.
This yields a compressed seven-lecture format.

For a slower or more advanced course,
Lecture 5 may be expanded into a full problem session,
and the sheaf- and topos-theoretic sections
may be treated in greater detail.
In that case,
the material can easily support ten or more lectures.

\subsection{Pedagogical remark}

A central pedagogical aim of the present notes
is that the reader should not merely learn the definitions,
but also see how those definitions are used.
For that reason,
the exercises and their detailed solutions
are an essential part of the design of the course.
They are intended to help students move
from formal familiarity with the definitions
to actual working understanding.

\subsection{A possible continuation: six lectures on the author's later paper}

The present notes were written not only as a self-contained introduction,
but also as preparatory background for later work of the author.
For that reason,
a natural use of the material is the following fifteen-lecture format:
nine lectures based on the present notes,
followed by six lectures devoted to the later paper
\cite{Inoue2026}.

In this format,
the first nine lectures provide the mathematical language
needed for the later theory:
group actions, torsors, affine spaces, local triviality,
gluing data, cocycles, sheaf torsors, and the basic topos-theoretic viewpoint.
The remaining six lectures then explain how these ideas
are used in the sheaf-theoretic interpretation
of $\Sigma$-protocols.

\subsection{A fifteen-lecture plan}

\paragraph{Lectures 1--9.}
These are the nine preparatory lectures described above.

\paragraph{Lecture 10: Motivation and global overview of the later paper.}
Explain the general purpose of the later work.
Why should protocol-theoretic data be organized sheaf-theoretically?
Why is locality important?
What is gained by passing from ordinary transcript sets
to sheaves on a site?
This lecture should help students see
that the later paper is not an unrelated abstraction,
but a natural continuation of the material developed in the first nine lectures.

\paragraph{Lecture 11: The site associated with the protocol.}
Introduce the site $(C_\Pi,J_\Pi)$ attached to a protocol $\Pi$.
Explain the meaning of its objects and morphisms,
and especially the role of coverings.
The main pedagogical aim is to clarify
what ``locality'' means in the protocol-theoretic setting,
and why a Grothendieck topology is the right formal tool
for encoding it.

\paragraph{Lecture 12: The protocol-dependent sheaf.}
Present the construction of the sheaf-like object $F_\Pi$.
Explain what its sections represent,
how restriction maps are interpreted,
and why sheaf-theoretic gluing is relevant.
This lecture should make clear that
one is no longer studying isolated transcripts,
but a system of local protocol data organized by compatibility.

\paragraph{Lecture 13: The symmetry object and the action.}
Introduce the corresponding group object or sheaf of groups $G_\Pi$,
and explain how it acts on $F_\Pi$.
At this point,
students should be encouraged to connect the formalism
with the earlier theory of ordinary torsors and sheaf torsors.
The main goal is to make the torsor claim
look structurally natural rather than surprising.

\paragraph{Lecture 14: The torsor theorem and proof strategy.}
State and explain the main torsor theorem of the later paper.
Discuss the role of local triviality,
the relation to simulation,
and the significance of the absence of suitable global sections.
If the proof is too long for a single lecture,
one may present a proof sketch and isolate the key ideas:
local sections as local trivializations,
transport under the symmetry object,
and global obstruction as a structural feature.

\paragraph{Lecture 15: Structural interpretation and further directions.}
Explain the resulting interpretation of security properties,
including the relation between local triviality
and honest-verifier zero-knowledge,
as well as the relation between global obstruction
and soundness-type phenomena.
Conclude by discussing possible extensions,
such as more general protocols,
richer topologies,
or broader uses of sheaf- and topos-theoretic methods in cryptography.

\subsection{Pedagogical value of the fifteen-lecture format}

This fifteen-lecture format has a clear conceptual advantage.
The first nine lectures provide a purely mathematical foundation
in which the students become familiar with the torsor idea
in increasingly sophisticated forms:
from affine spaces and group actions
to gluing data, cocycles, sheaf torsors, and topoi.
The last six lectures then show
that these notions are not merely formal generalities,
but tools that can actually organize cryptographic structure.

In this way,
the later paper may be taught not as an isolated advanced text,
but as the natural application of a well-prepared mathematical background.
This makes the conceptual content of the later theorem
far easier to understand.

\subsection{A practical remark for instructors}

For a mathematically mature audience,
the fifteen-lecture course may be taught almost exactly as written.
For a less experienced audience,
it may be helpful to slow down Lectures 8--11,
since the transition from elementary torsors
to sheaf torsors and then to sheaves on a site
is the point at which the abstraction level rises most sharply.

It is also useful, in the second half of the course,
to refer back explicitly to the first half.
For example:
\begin{itemize}
\item local sections in the later paper should be compared
with local basepoints in sheaf torsors;
\item protocol-dependent transition data should be compared
with the cocycles of Sections~8 and~9;
\item the absence of suitable global sections should be compared
with the failure of global triviality for a torsor.
\end{itemize}
Such reminders help students see
that the later theory is genuinely built from the earlier one.

\subsection{Final pedagogical remark}

The overall teaching strategy behind this course is therefore cumulative.
The first nine lectures train the reader
to understand torsors as spaces of transport,
as locally trivial objects,
and as globally glued structures.
The final six lectures then show
how that mathematical language may be used
to reinterpret protocol-theoretic notions
in a sheaf-theoretic and topos-theoretic framework.
This is, in fact, one of the main intended uses of the present notes.

\section{Glossary of Basic Terms}

This glossary is intended only as a quick guide to the main terms
used throughout these notes.
The explanations are informal and are meant
to recall the role of each notion in the present exposition.

\begin{description}

\item[Affine space.]
A nonempty set modeled on a vector space,
in which one can add a vector to a point
and subtract two points to obtain a vector,
but in which no origin is canonically distinguished.
In these notes, affine spaces are the guiding example of torsors.

\item[Basepoint.]
A chosen point of a torsor.
Once a basepoint is chosen,
the torsor becomes identified with the acting group,
but only non-canonically.

\item[Cocycle.]
A compatible system of transition data on overlaps.
For torsors, cocycles record how local trivializations differ
and how they glue to form a global object.

\item[Descent.]
The general principle that a global object may be reconstructed
from local objects together with compatible gluing data.
Torsors are among the basic examples of descent objects.

\item[Free action.]
A group action in which no non-identity group element fixes a point.
Equivalently, every stabilizer is trivial.

\item[Global section.]
A globally defined section of a sheaf or sheaf-like object.
For a sheaf torsor, a global section gives a global trivialization.

\item[Gluing data.]
The data specifying how local trivial pieces are identified on overlaps.
For torsors, this is expressed by transition elements or cocycles.

\item[Group action.]
An action of a group $G$ on a set or object $X$,
describing how elements of $G$ move points of $X$.

\item[Group object.]
The internal version of a group inside a category such as a topos.
In a sheaf topos, group objects correspond to sheaves of groups.

\item[Local section.]
A section defined only on a local region.
For a sheaf torsor, a local section plays the role of a local basepoint.

\item[Local triviality.]
The property that an object becomes trivial after restricting to sufficiently small regions.
For torsors, this means that locally they look like the acting group.

\item[Orbit.]
The set of points that can be reached from a given point
under a group action.

\item[Principal homogeneous space.]
Another name for a torsor.
This terminology emphasizes that every point looks the same under the action
and that there is no preferred origin.

\item[Sheaf of groups.]
A sheaf whose sections over each open set form a group,
with restriction maps given by group homomorphisms.

\item[Sheaf torsor.]
A torsor under a sheaf of groups.
It is locally nonempty and locally free and transitive,
but may fail to have a global section.

\item[Stabilizer.]
For a point $x$ under a group action,
the subgroup consisting of all elements that fix $x$.

\item[Topos.]
A mathematical universe of generalized spaces or local data,
such as a category of sheaves on a site.
In these notes, topoi provide the natural ambient setting
for sheaf torsors and later applications.

\item[Torsor.]
A nonempty object with a free and transitive action of a group,
or more generally of a sheaf of groups.
A torsor may be understood as a space with symmetry
but without a chosen origin.

\item[Transitive action.]
A group action in which every point can be moved to every other point.
Equivalently, the action has a single orbit.

\item[Transporter.]
For two points $x,y$ in a torsor,
the unique group element that sends $x$ to $y$.
It plays the role of a generalized difference.

\end{description}

\end{document}